\documentclass[a4paper,11pt,oneside,openany,reqno]{amsart}

\usepackage{latexsym,amsmath,amssymb,amsfonts}
\usepackage{bm}
\usepackage[dvipdfmx]{graphicx}
\usepackage{subfigure}
\usepackage{verbatim}
\usepackage{wrapfig}
\usepackage{ascmac}
\usepackage{cases}
\usepackage{color}

\usepackage{comment}
\usepackage{ulem}
\usepackage {cancel}

  \makeatletter
    
    \@addtoreset{equation}{section}
  \makeatother
\newtheorem{theo}{Theorem}[section]
\newtheorem{defi}[theo]{Definition}
\newtheorem{lemm}[theo]{Lemma}

\newtheorem{rem}[theo]{Remark}

\newcounter{c}
\newcounter{d}
\newcounter{b}
\newcommand{\cc}[1][]{\refstepcounter{c}#1\arabic{c} } 
\newcommand{\ce}[1][]{\refstepcounter{d}#1\arabic{d}}

\newcommand{\res}{\mathop{\hbox{\vrule height 7pt width 
0.5pt depth 0pt \vrule height 0.5pt width 6pt depth 0pt}}\nolimits}

\newcommand{\Na}{\mathbb{N}} 
\newcommand{\Z}{\mathbb{Z}} 
\newcommand{\R}{\mathbb{R}} 
\newcommand{\T}{\mathbb {T}} 

\newcommand{\e}{\varepsilon} 
\newcommand{\vhi}{\varphi} 
\newcommand{\spt}{\mathrm{spt}} 
\newcommand{\dist}{\mathrm{dist}} 

\title[Convergence of the Allen-Cahn equation]{Convergence of the Allen-Cahn equation with transport term in a bounded domain}
\author{Yuki Tsukamoto}

\address{
	Organization for the Strategic Coordination of Research and Intellectual Properties, Meiji University, 4-21-1 Nakano,
	Nakano-ku, Tokyo, 164-8525, Japan.}

\begin{document}
\maketitle
\begin{abstract}	
	We study the Allen--Cahn equation with respect to a transport term in a bounded domain.
	We prove that the limit interface is the mean curvature flow with the transport term, given the condition that the energy is uniformly bounded with respect to time. Using this result, we show the existence of the mean curvature flow with a gradient vector field as the transport term.
\end{abstract}

\section{Introduction}
In this study, we consider the following Allen--Cahn equation \cite{SA79} with a transport term:
\begin{align}
	\partial_t \vhi_\e + u_\e\cdot \nabla \vhi_\e = \Delta \vhi_\e - \frac{W'(\vhi_\e)}{\e^2 } \quad
	\mathrm{on} \  \Omega, \label{ae1}
\end{align}
where $\Omega \subset \R^n$ is a bounded domain with smooth boundary,
$\e$ is a small parameter, $u_\e:\Omega \times[0,\infty) \to \R^n$ is a given vector field, and $W$ is the double-well potential, such as $W(s)=(1-s^2)^2/2$. The term $u_\e\cdot \nabla \vhi_\e$ is referred to as the transport term.

In the case of $u_\e=0$, the Allen--Cahn equation has been studied by many researchers
with different settings and assumptions. Ilmanen \cite{Il93} proved that if $\e$ approaches zero, $\Omega$ is equal to $\R^n$, and the appropriate initial value, then $\vhi_\e$ is a two-phase flow, that is, $\vhi_\e(\cdot,t)$ is equal to $\pm 1$ almost everywhere for each $t\geq 0$.
Moreover, when the measure $\mu^\e_t$, the interface of $\vhi_\e$ is defined as
\begin{align}
	\mu^\e_t := \left( \frac{\e |\nabla \vhi_\e(\cdot,t)|^2}{2} + \frac{W(\vhi_\e(\cdot,t))}{\e}\right) \mathcal{L}^n\lfloor_\Omega.
\end{align}
The limit measure $\mu^0_t$ of $\mu^\e_t$ is approximately an $(n-1)$-dimensional surface and satisfies the mean curvature flow $v=h$  in the sense of Brakke \cite{BR78}. Here, $v$ is the velocity vector, and $h$ is the mean curvature vector of $\spt  \mu^0_t$. In the case of a bounded domain $\Omega$, Giga, Onoue, and Takasao \cite{GI21} studied the Dirichlet boundary condition and dynamic boundary condition, and some researchers \cite{KA19, KA95,MI15, MO21} studied the Neumann conditions.
For the general vector field $u_\e$, Liu, Sato, and Tonegawa \cite{LI10}
proved that the limit measure of $\mu^\e_t$
satisfies the mean curvature flow with a transport term $v=h+(\nu\cdot u)\nu$
on a two-dimensional torus $\T^2$ and a three-dimensional torus $\T^3$ under the condition that the  vector field $u_\e$ belongs to the appropriate function space, and the initial energy $\mu^\e_0(\Omega)$ and the initial upper density ratio
\begin{align*} 
	\sup_{B_r(x)\subset \Omega} \frac{\mu^\e_0(B_r(x))}{\omega_{n-1} r^{n-1}}
\end{align*} are uniformly bounded independent of $\e$. Here, $u$ is the limit vector field of $u_\e$, $\nu$ is the unit normal vector of $\spt  \mu^0_t$ and $\omega_{n-1}:=\mathcal{L}^{n-1} (B_1(0) )$.
Under the assumption that the initial energy and the initial upper density ratio
are uniformly bounded, and $u_\e$ belongs to the function space different from \cite{LI10},
 Takasao and Tonegawa \cite{TA16} also proved the existence of limit measure of $\mu^\e_t$ on a domain without boundary, such as a torus $\T^n$ or $\R^n$.
Under the same conditions for the energy and the upper density ratio, Takasao \cite{TA20} proved the existence of a weak solution for the mean curvature flow with transport term and forcing term $v=h+(u \cdot \nu +g)\nu$ by considering the \eqref{ae1} with the addition of the forcing term on $\T^n$.
In \cite{TA16, TA20, LI10}, the initial surface of the mean curvature flow must have finite the upper density ratio and the perimeter.

Herein, we prove that $\mu^\e_t$ converges under different assumptions from previous works, where $\Omega\subset \R^n$ is a bounded domain, and the energy is uniformly bounded independent of $\e$ and $t$. It is difficult to obtain an estimate of the upper density ratio in a neighbourhood of the boundary $\Omega$, so many papers consider it in a no boundary domain such as $\R^n$ or $\T^n$. 
However, by considering a relatively compact domain, we have successfully estimated the upper density ratio on that domain.
Moreover, since we do not assume the boundedness of the initial upper density ratio, we can show that
for an $(n-1)$-dimensional surface $\Gamma \subset \Omega$, not necessarily the $C^1$-class,
there exists a time-evolving surface $\{\Gamma_t\}_{t\geq 0}$ such that it satisfies
the mean curvature flow with a transport term as mentioned by Brakke, and $\Gamma_0$
is equal to $\Gamma$. To prove this theorem, the estimate of the upper-density ratio is
obtained by extending the result of the steady-state Allen--Cahn equation of \cite{TO19}.
Using this estimate, we can show a monotonic formula, which is an important tool for
 geometric measure theory, and then use the result of \cite{TA16}.
\subsection{The associated varifolds}
We associate to each solution of \eqref{ae1} a varifold in the following. We refer to \cite{WA72, LS83} for a comprehensive treatment of varifolds. Let $\mathbf{G}(n,n-1)$ be
the apace of $(n-1)$-dimensional subspaces of $\R^n$. We consider $S \in \mathbf{G}(n,n-1)$
as the $n \times n$ matrix representing the orthogonal projection of $\R^n$ onto $S$.
For $n \times n$ matrices $A$ and $B$, we define
\begin{align*}
A \cdot B := \mathrm{trace} (A^t \circ B) = \sum_{i,j=1}^{n} A_{ij} B_{ij},
\end{align*}
where $\circ$ is the matrix multiplication. A set $V$ is called an $(n-1)$-dimensional
varifold in $\Omega \subset \R^n$ if $V$ is a Radon measure on $\mathbf{G}_{n-1} (\Omega) :=
\Omega \times \mathbf{G}(n,n-1)$. We define $\mathbf{V}_{n-1}(\Omega)$ to be the set of all
$(n-1)$-dimensional varifold in $\Omega$.
Convergence in the varifold sense means convergence
in the usual sense of measures. For $V \in \mathbf{V}_{n-1}(\Omega)$, let $\| V\|$
be the weight measure of $V$.
For $V \in {\bf V}_{n-1}(\Omega)$, we define the first variation of $V$ by
\begin{align}
	\label{4re-eq7}
	\delta V(g) := \int_{{\bf G}_{n-1}(\Omega)} \nabla g(x) \cdot S \  dV(x,S)
\end{align}
for any vector field $g \in C^1_c(\Omega;\mathbb{R}^n)$.
Let $\|\delta V\|$ be the total variation of $\delta V$. 

We associate each function $\vhi_\e(\cdot,t)$ with a varifold $V^\e_t$ as follows.
Let $V^\e_t \in {\bf V}_{n-1}(\Omega)$ be defined by
\begin{align}
	V^\e_t (\phi) := \int_{\{ |\nabla \vhi_\e(\cdot,t)| \neq 0 \}} \phi \Big( x,I-\frac{\nabla \vhi_\e(x,t)}{|\nabla \vhi_\e(x,t)|} \otimes
	\frac{\nabla \vhi_\e(x,t)}{|\nabla \vhi_\e(x,t)|} \Big) d\mu_t^\e(x)
	\label{4defvari}
\end{align}
for $\phi \in C_c ({\bf G}_{n-1}(\Omega))$, where $I$ is the $n \times n$ identity matrix
and $\otimes$ is the tensor product.
By definition of $V^\e_t$, we have
\begin{align*}
\|V^\e_t\|=\mu_t^\e\res_{\{| \nabla \vhi_\e(\cdot,t)|\neq 0\}}
\end{align*}
and by \eqref{4re-eq7}, we obtain
\begin{align}
	\delta V_t^\e (g) = \int_{\{ |\nabla \vhi_\e(\cdot,t)| \neq 0 \}} \nabla g \cdot
	\Big( I-\frac{\nabla \vhi_\e(x,t)}{|\nabla \vhi_\e(x,t)|} \otimes
	\frac{\nabla \vhi_\e(x,t)}{|\nabla \vhi_\e(x,t)|} \Big) d\mu_t^\e
	\label{4defvari1}
\end{align}
for each $g\in C^1_c(\Omega,\mathbb{R}^n)$.
\begin{defi}\label{def1}
	A family of varifolds $\{V_t\}_{t>0} \subset \mathbf{V}_{n-1(\Omega)}$ is a generalized solution of $v=h+(\nu\cdot u)\nu$ if the following conditions are satisfied.
	\begin{itemize}
		\item[(a)] $V_t \in \mathbf{IV}_{n-1}(\Omega)$ for a.e. $t>0$, where $\mathbf{IV}_{n-1}(\Omega)$ is the set of all integral $(n-1)$-varifolds in $\Omega$.
		\item[(b)] For all $0<\tau<T$,
		\begin{align}
			\sup_{t \in [\tau,T]} \|V_t \|(\Omega) +
			\sup_{t \in [\tau,T],B_r(x)\subset \Omega} \frac{\|V_t\|(B_r(x))}{\omega r^{n-1}} <\infty.
		\end{align}
		\item[(c)] For all $0<\tau<T$,
		\begin{align}
			\int^T_\tau \ dt \int_\Omega |h|^2+|u|^2 \ d \|V_t\|<\infty.
		\end{align}
		\item[(d)] For all $\phi \in C^1_c (\Omega \times [0,\infty);\R^+)$ and $0<t_1<t_2<\infty$,
		\begin{align}
			\left.\|V_t\|(\phi(\cdot,t))\right|^{t_2}_{t=t_1} \leq \int^{t_2}_{t_1} \ dt \int_{\Omega}
			(\nabla \phi -h \phi)\cdot \{h+(u\cdot \nu)\nu\} + \partial_t \phi \ d\|V_t\|.
		\end{align}
	\end{itemize}
\end{defi}
Note that this definition is slightly different from the definition of \cite{TA16} because it does not include $t=0$.
\subsection{Main Theorems}
The following theorem is the main result of this study.
Throughout this study, we assume that $\Omega \subset \R^n$ is a bounded domain with
smooth boundary $\partial \Omega$. Suppose that $W \in C^3 (\R)$ satisfies the following:
\begin{itemize}
	\item[(W-a)] The function $W$ has two strict minima $W(\pm 1)=W'(\pm 1)=0$.
	\item[(W-b)] For some $\gamma \in(-1,1)$, $W'>0$ 
	on $(-1,\gamma)$ and $W'<0$ on $(\gamma,1)$.\label{test}
	\item[(W-c)] For some $\alpha \in (0,1)$ and $\kappa>0$, $W''(x)\geq \kappa$ for all $|x| \geq \alpha$.
\end{itemize}
\begin{theo} \label{MT1}
	Suppose $n \geq 2$, 
	\begin{align} \label{MTe1}
	2<q<\infty, \ \frac{nq}{2(q-1)}<p<\infty \ \left( \frac{4}{3} \leq p<\infty
	 \ \mathrm{in \ addition \ if} \ n=2 \right),
	\end{align}
	$0<\beta<1/2$, $0<\e<1$, and $\vhi_\e$ satisfies
	\begin{align}
	\partial_t \vhi_\e + u_\e\cdot \nabla \vhi_\e = \Delta \vhi_\e - \frac{W'(\vhi_\e)}{\e^2} \ \mathrm{on}
	 \ \Omega \times[0,T]. \label{MTe2} 
	\end{align}
	Assume $\Omega' \subset \subset \Omega$, $0<\tau<T $, $u_\e \in C^\infty (\Omega \times [0,T])$, $\nabla^j \vhi_\e$, 
	$\partial_t \nabla^k \vhi_\e \in C(\Omega \times [0,T])$ for $k\in\{0,1\}$ and $j \in \{0,1,2,3\}$, respectively.
	Let $\mu^\e_t$ be a Radon measure on $\Omega$ defined by
	\begin{align}
	\int_{\Omega} \phi(x) \ d\mu^\e_t(x) := 
	\int_{\Omega} \phi(x) \left( \frac{\e|\nabla \vhi_\e(x,t)|^2}{2} +
	\frac{W(\vhi_\e(x,t))}{\e} \right)  \ dx
	\end{align}
	for $\phi \in C_c(\Omega)$. Assume
\begin{align}
	\sup_{\Omega \times [0, T]} |\vhi_\e| \leq 1, \label{a2} \\
	\sup_{\Omega \times [0, T]} |u_\e| \leq \e^{-\beta} , 
	\sup_{\Omega \times [0, T]} |\nabla u_\e| \leq \e^{-(\beta+1)}, \label{a3} \\
	\|u_\e\|_{L^q([0,T];(W^{1,p}(\Omega))^n)} \leq \Lambda_0, \label{a4} \\
	\sup_{[0, T]} \mu_t^\e(\Omega) \leq \Lambda_1. \label{a5}
\end{align}
	Then, there exist constants $\e_{\ce \label{ME1}}= \e_{\ref{ME1}}(n,p,q,\beta,\Lambda_0,\Lambda_1,\tau, \Omega', \Omega,W)>0$, and
	 $D_1= D_1(n,p,q,\beta,\Lambda_0,\Lambda_1,\tau, \Omega', \Omega,W)>0$ such that, if
	$\e<\e_{\ref{ME1}}$, we have
	\begin{align}
	\sup_{U_{r}(x)\subset \Omega',t \in [\tau,T]}  \frac{\mu_{t}^\e (B_r(x))}{r^{n-1}} \leq D_1.
	\end{align}
\end{theo}
We can show the following by Theorem \ref{MT1}, \cite{TA16}, and some arguments.
\begin{theo} \label{MT2}
	Suppose $n \geq 2$,
	\begin{align} 
	2<q<\infty, \ \frac{nq}{2(q-1)}<p<\infty \ \left( \frac{4}{3} \leq p<\infty
	\ \mathrm{in \ addition \ if} \ n=2 \right).
	\end{align}
	Given any $g \in L^q([0,\infty);W^{2,p}(\Omega)) \cap W^{1,\infty}([0,\infty);L^\infty(\Omega))$ and a non-empty domain
	$\Omega_0 \subset \subset \Omega$ with $	\chi_{\Omega_0} \in BV(\Omega)$ and $M_0=\partial \Omega_0$, we have the following properties:
	\begin{enumerate}
		\item There exists a family of varifolds $\{V_t\}_{t>0} \subset {\bf V}_{n-1}(\Omega)$
		which is a generalized solution of $v= h +(\nabla g \cdot \nu)\nu$, as in 
		Definition \ref{def1}.
		\item There exists a function $\vhi \in BV_{loc}(\Omega \times [0,\infty)) \cap C^{\frac12}_{loc}
		((0,\infty);L^1(\Omega))$ with the following properties.
\begin{itemize}
	\item[(2-a)] $\vhi(\cdot,t)$ is a characteristic function for all $t \in [0,\infty)$.
	\item[(2-b)]  $\|\nabla \vhi(\cdot,t)\| (\phi) \leq \|V_t\|(\phi) $ for all
	$t\in(0,\infty)$ and $\phi \in C_c(\Omega;\R^+)$.
	\item[(2-c)]  $\vhi(\cdot,0) = \chi_{\Omega_0}$ a.e. on $\Omega$.
	\item[(2-d)]  There exist constants $\iota=\iota(n,p,q)>0$ and $c_{\cc \label{Mc1}} =c_{ \ref{Mc1}}(n,p,q,\beta,g,\Omega_0,  \Omega,W)$ such that for any $0<t<1$, we have
	\begin{align*}
		|\Omega_0 \Delta \Omega_t|\leq c_{ \ref{Mc1}} t^\iota,
	\end{align*}
where $\Omega_t = \{x\in \Omega; \vhi(x,t)=1 \}$.
	\item[(2-e)]  Writing $\|V_t\|= \theta \mathcal{H}^{n-1}\lfloor_{M_t}$ and
	$\| \nabla \vhi(\cdot,t)\| = \mathcal{H}^{n-1}\lfloor_{\tilde{M}_t}$ for a.e. $t>0$,
	we have 
	\begin{align}
	\mathcal{H}^{n-1}(\tilde{M}_t \backslash M_t)=0
	\end{align}
	and
	\begin{align}
	\theta(x,t)= \begin{cases}
	\mathrm{even \ integer} \geq 2 \quad \mathrm{if} \ x \in M_t \backslash \tilde{M}_t ,\\
	\mathrm{odd \  integer} \geq 1 \quad \mathrm{if} \ x \in \tilde{M}_t 
	\end{cases}
	\end{align}
	for $\mathcal{H}^{n-1}$ a.e. $x \in M_t$.
\end{itemize}
\item If $p<n$, then for any $T>\tau>0$, setting $s:=\frac{p(n-1)}{n-p}$, we have
\begin{align}\label{MT2e1}
\left(\int^T_\tau \left(\int_{\Omega} |\nabla g|^s \ d \|V_t\| \right)^{\frac{q}{s}} dt \right)^{\frac{1}{q}}<\infty.
\end{align}
If $p=n$, we have \eqref{MT2e1} for any $2\leq s<\infty$.
	\end{enumerate}
\end{theo}
\begin{rem}
	Unlike \cite{TA16}, we cannot prove whether $V_t$ has unit density because the density of the initial surface is not necessarily bounded. We consider that the initial surface condition is satisfied by (2-d).

	
	Even if $\Omega=\R^n$, Theorem \ref{MT2} holds because $\Omega_0$ is bounded.
	But when $\Omega_0$ is not bounded, the theorem does not hold. Also, this condition is not considered in \cite{TA16}.
	
	It is difficult to check whether the assumption \eqref{a5} is satisfied for the general vector field $u$. However, when $u = \nabla g$, we can show that \eqref{a5} is satisfied by Lemma \ref{l3a}.
	
	The function $\vhi$ appearing in (2) of Theorem \ref{MT2} consists in the following.
	Let $\vhi_{\e_i} \in C^\infty(\Omega\times [0,i))$ the solution of \eqref{l31}.
	Define
	\begin{align*}
		w_i := \Phi \circ \vhi_{\e_i} \ \mathrm{with} \ \Phi(s):= \frac{\int^s_{-1}\sqrt{2W(y)} \ dy}{\int^1_{-1}\sqrt{2W(y)} \ dy}.
	\end{align*}
Then there exists a subsequence $\{w_{i_k}\}_{k\in \Na} \subset\{w_i\}_{i\in \Na}$
which converges $L^1$ and pointwise, and the limit function is $\vhi$.
\end{rem}

	In two dimensions, an example of a finite perimeter but infinite density include the following. 
	\begin{lemm}
		There exists a function $f \in C([0,1])$ such that
		\begin{align*}
			\mathcal{H}^1\lfloor_{A}(\R^2) <\infty,\\
			\limsup_{r \to 0} \frac{\mathcal{H}^1 \lfloor_{A } (B_r)}{r} =\infty,
		\end{align*}
	where $A= \{(x,y) \in [0,1]\times \R; y=f(x) \}$.
	\end{lemm}
\begin{proof}
The function $f \in C([0,1])$ is defined by a line connecting points 
$(k^{-1}, (-1)^k k^{-3/2})$ and $((k+1)^{-1}, (-1)^{k+1} (k+1)^{-3/2})$ on $[(k+1)^{-1},k^{-1}]$, where $f(0)=0$.
The length of the line on $[(k+1)^{-1},k^{-1}]$ is as follows.
\begin{align*}
	&\mathcal{H}^1\lfloor_{A}([(k+1)^{-1},k^{-1}] \times \R)\\
	=& \sqrt{\left( \frac{1}{k(k+1)}\right)^2 +\left(\frac{1}{k^{\frac{3}{2}}}
		+\frac{1}{(k+1)^{\frac{3}{2}}}\right)^2 } \\
	\leq& \frac{3}{k^{\frac{3}{2}}}.
\end{align*}
Hence, we obtain
\begin{align*} 
		\mathcal{H}^1\lfloor_{A}(\R^2) \leq \sum_{k=1}^{\infty} \frac{3}{k^{\frac{3}{2}}}<\infty.
\end{align*}
Since $\mathcal{H}^1\lfloor_{A}([(k+1)^{-1},k^{-1}] \times \R) \geq k^{-3/2}$, we have
\begin{align*}
	\limsup_{r \to 0}\frac{\mathcal{H}^1 \lfloor_{A } (B_r)}{r} &\geq \limsup_{r \to 0} \frac{1}{r}
	\sum_{k=\lfloor r^{-1} \rfloor }^{\infty} \frac{1}{k^{\frac{3}{2}}}\\
	&\geq \limsup_{r \to 0}\frac{1}{r} \int^\infty_{\lfloor r^{-1} \rfloor } x^{-\frac{3}{2}} \ dx\\
	&\geq \limsup_{r \to 0}\frac{2(\lfloor r^{-1} \rfloor )^{-\frac{1}{2}}}{r}\\
	&= \infty,
\end{align*}
where $\lfloor x \rfloor = \max \{m \in \Z ;m\leq x \} $.
Thus this lemma follows.
\end{proof}
If part of the boundary of $\Omega_0$ is such the function,
it cannot be treated in previous studies,
 but it is an example that can be treated by this result.

The remainder of this paper is organized as follows. In Section 2, we first obtain an estimate of the upper density ratio
on the open unit ball. Using this result, we prove Theorem \ref{MT1} and the theorems required in
later section. In Section 3, we construct functions that satisfy the assumptions of Theorem \ref{MT1} and are associated with a given initial surface. We then show that the results of \cite{TA16} are valid, and prove Theorem \ref{MT2}.

	\section{The estimate for the upper density ratio}

Throughout this section, we assume \eqref{MTe1}--\eqref{a5} and  we drop the index $\e$, that is, we write $u$, $\vhi$, and $\mu_t$ instead of $u_\e$, $\vhi_\e$, and $\mu_t^\e$.
Up to Theorem \ref{l9}, the set $\Omega= U_1=\{|x|<1\}$ since the result is local.

\begin{lemm} \label{l1}
	There exists a constant $c_{\cc \label{1c1}} = c_{\ref{1c1}}(n,  W)$ such that
	\begin{align} 
	\sup_{U_{1-\e} \times [\e^2, T]} \e|\nabla \vhi| +
	\sup_{x,y \in U_{1-\e} ,t \in[\e^2,T]} \e^{\frac32}\frac{|\nabla \vhi(x,t) - \nabla \vhi (y,t)|}{|x-y|^\frac12} \leq c_{\ref{1c1}} \label{1e1}
	\end{align}
	for $0<\e <1/8$. If $\e \geq 1/8$, then we have for any $0<d<1$
	\begin{align}
	\sup_{U_{d} \times [\frac14, T]} \e|\nabla \vhi| 
	 \leq c_{\ref{1c1}} \label{1e2}
	\end{align}
	where $c_{\ref{1c1}}$ depends additionally on $d$.
\end{lemm}
\begin{proof}
	We first consider the case of $0<\e <1/8$. Take any domain $B_{2\e}(x_0) \times [t_0-2\e^2,t_0]
	\subset U_{1} \times [0,T]$. Define $\tilde{\vhi}(x,t):= \vhi(\e x+x_0,\e^2 t+t_0-2\e^2)$
	and $\tilde{u}(x,t):=u(\e x+x_0,\e^2 t+t_0-2\e^2)$ for $(x,t )\in B_2 \times [0,2]$.
	By \eqref{a3}, we obtain
	\begin{align} \label{l1e2}
	\sup_{B_2 \times [0,2]} (|\tilde{u}| + |\nabla \tilde{u}|) \leq \e^{-\beta}.
	\end{align}
	By \eqref{MTe2}, we have
	\begin{align} \label{l1e1}
	\partial_t \tilde{\vhi} +\e \tilde{u}\cdot \nabla \tilde{\vhi}
	= \Delta \tilde{\vhi} -W'(\tilde{\vhi}).
	\end{align}
	Using the interior estimate of \cite[p. 342, Theorem 9.1]{LA68}, if $0<s_1<s_2<\frac32$, $0\leq \tau_2<\tau_1<2$, and $\partial_t v-\Delta v = f$ on $B_\frac32 \times [0,2]$,
	then we have
	\begin{align} \label{l1e4}
	&\| \partial_t v, \nabla^2 v\|_{L^r(B_{s_1}\times [\tau_1,2])} \nonumber\\ \leq& c(n,r,s_1,s_2,\tau_1,\tau_2) (\|f\|_{L^r(B_{s_2}\times [\tau_2,2])}
	+\|v\|_{L^r(B_{s_2}\times [\tau_2,2])}  )
	\end{align}
	for $r \in (1,\infty)$. Let $\phi \in C^1_c (B_\frac32)$ be a cut-off function and multiply $\phi^2 \tilde{\vhi}$ to \eqref{l1e1}, then by integration by parts, \eqref{a2}, and \eqref{l1e2}, we obtain
	\begin{align} \label{l1e3}
	\int^2_0 \int_{B_\frac32} |\nabla \tilde{\vhi}|^2 \ dxdt \leq c(W).
	\end{align}
	By \eqref{a2}, \eqref{l1e2}, \eqref{l1e4}, and \eqref{l1e3}, we have
	\begin{align} \label{l1e6}
	&\int^2_{\frac23} \int_{B_{\frac43}} |\tilde{\vhi}_t|^2+|\nabla^2 \tilde{\vhi}|^2 \ dxdt \\
	\leq& c(n)(\|\e \tilde{u}\cdot \nabla \tilde{\vhi} +W'(\tilde{\vhi})\|_{L^2(B_\frac32\times [\frac12,2])}+
	\|\tilde{\vhi} \|_{L^2(B_\frac32\times [\frac12,2])}  )\nonumber \\
	\leq& c(n,W).
	\end{align}
	Differentiate \eqref{l1e1} with respect to $x_j$, we have
	\begin{align}\label{l1e5}
	\partial_t(\tilde{\vhi}_{x_j} ) - \Delta \tilde{\vhi}_{x_j}
	= -\e \tilde{u}_{x_j} \cdot \nabla \tilde{\vhi} -\e \tilde{u}\cdot \nabla \tilde{\vhi}_{x_j}
	-W''(\tilde{\vhi})\tilde{\vhi}_{x_j}.
	\end{align}
	Using \eqref{a2}, \eqref{l1e2}, \eqref{l1e4}, \eqref{l1e3}, and \eqref{l1e6}, we obtain
	\begin{align}\label{l1e7}
	\int^2_\frac34 \int_{B_{\frac54}} |\nabla \tilde{\vhi}_t|^2+|\nabla^3 \tilde{\vhi}|^2 \ dxdt \leq c(n,r,W).
	\end{align}
	By \eqref{l1e7} and the Sobolev inequality, we have
	\begin{align}
	\| \nabla \tilde{\vhi} \|_{L^{\frac{2(n+1)}{n-1}}(B_{\frac54} \times [\frac34,2])}
	\leq c(n,W).
	\end{align}
	We can use this estimate to \eqref{l1e1} and \eqref{l1e4} with $r=\frac{2(n+1)}{n-1}$.
	By repeating this argument a finite number of times, we obtain
	\begin{align}
	\| \nabla \tilde{\vhi} \|_{L^{2n}(B_{\frac{k_1+1}{k_1}} \times [\frac{k_2}{k_2+1},2])}
	\leq c(n,W),
	\end{align}
	where $k_1,k_2>0$ are constants. Using Morrey's inequality, we have
	\begin{align}
	\| \nabla \tilde{\vhi} \|_{C^{\frac12}(B_1 \times [1,2])} \leq
	\| \nabla \tilde{\vhi} \|_{L^{2n}(B_1 \times [1,2])}\leq c(n,W),
	\end{align}
	and \eqref{1e1} follows. The case of $\e \geq 1/8$ does not require the change of variables as above and the proof is omitted.
\end{proof}
In the following, define $\tilde{\beta} = \frac{1}{2}+\frac{\beta}{4}$. We have $\beta<\tilde{\beta}<1$
by definition of $\beta$.
\begin{lemm} \label{l2}
	Given $0<s<1$, there exist constants $c_{\cc \label{2cc1}}=c_{\ref{2cc1}}(\beta)$ an $\e_{\ce \label{2E1}} =\e_{\ref{2E1}}(n,W,\beta,s)$ such that,
	if $\e< \e_{\ref{2E1}}$,
	\begin{align} \label{l2e1}
	\frac{\e |\nabla \vhi|^2}{2} -\frac{W(\vhi)}{\e} \leq c_{\ref{2cc1}} \e^{-\beta} \ \mathrm{on} \ 
	U_{s}\times [\e^{2\tilde{\beta}}, T]. 
	\end{align}
\end{lemm}
\begin{proof}
	Define $\tilde{\vhi}(x,t):= \vhi(\e x,\e^2 t)$
	and $\tilde{u}(x,t):=u(\e x,\e^2 t)$ for $(x,t )\in U_{\e^{-1}} \times [0,\e^{-2}T]$, and subsequently drop $\tilde{\cdot}$ for simplicity. Define $W :=W(\vhi)$, $G:=G(\vhi)$, and
	\begin{align}
	\xi:=\frac{|\nabla \vhi|^2}{2}-W(\vhi) -G(\vhi),
	\end{align}
	where $G$ will be chosen later. By the proof of \cite[Lemma 4.2]{TA16}, we obtain
	on $\{ |\nabla \vhi|>0 \}$
	\begin{align}
	\partial_t \xi + \e u \cdot \nabla \xi -\Delta \xi \leq&
	-(G')^2-W'G'+2G''(\xi +W+G) \nonumber\\
	&-\frac{2(W'+G')}{|\nabla \vhi|^2} \nabla \xi \cdot \nabla \vhi
	-\e \nabla u \cdot(\nabla \vhi \otimes \nabla \vhi). \label{l2e2}
	\end{align}
	Define $s_1:= (1-s)/3$ and $M:=\sup_{U_{\e^{-1}(s+2s_1)} \times [\e^{-2(1-\tilde{\beta})}/4,\e^{-2}T]} \left(\frac{|\nabla \vhi|^2}{2} - W(\vhi) \right)$.
	Note that $M$ is bounded depending only on $n$ and $W$ by Lemma \ref{l1}.
	Since $M \leq 0$ implies \eqref{l2e1}, we may assume $M>0$.
	Let $\phi(x,t) \in C^\infty (U_{\e^{-1}(s+2s_1)} \times [\e^{-2(1-\tilde{\beta})}/4,\e^{-2}T])$ be such that
	$\phi =0$ on $\ U_{\e^{-1}s} \times [3\e^{-2(1-\tilde{\beta})}/4, \e^{-2}T]$,
	$\phi = M$ on 
	$ U_{\e^{-1}(s+2s_1)} \times [\e^{-2(1-\tilde{\beta})}/4,\e^{-2}T] \backslash
	U_{\e^{-1}(s+s_1)} \times [2 \e^{-2(1-\tilde{\beta})}/4,\e^{-2}T]$, $0\leq \phi \leq M$, and
	\begin{align} \label{l2e3}
	  |\nabla \phi| \leq 2s^{-1}_1\e M, \ |\Delta \phi |\leq 4n s_1^{-2}\e^2 M, \ 
	\partial_t \phi \geq -8\e^{2(1-\tilde{\beta})}M.
	\end{align}
	Let
	\begin{align*}
	\tilde{\xi}&:=\xi-\phi,\\
	G(\vhi)&:= 8\e^{b} \left(1-\frac18 (\vhi- \gamma)^2\right),
	\end{align*}
	where $\gamma$ is as in (W-b) and $b$ is a positive constant less than or equal to $1-\beta$. To derive a contradiction, suppose that
	\begin{align*}
	\sup_{\ U_{\e^{-1}s} \times [3\e^{-2(1-\tilde{\beta}) }/4, \e^{-2}T]} \xi \geq \e^{b}.
	\end{align*}
	By definition of $\tilde{\xi}$, we have $\tilde{\xi} \leq 0$ on $U_{\e^{-1}(s+2s_1)} \times [\e^{-2(1-\tilde{\beta})}/4,\e^{-2}T] \backslash
	U_{\e^{-1}(s+s_1)} \times [2 \e^{-2(1-\tilde{\beta})}/4,\e^{-2}T]$ and 
	\begin{align*}
\sup_{\ U_{\e^{-1}s} \times [3\e^{-2(1-\tilde{\beta}) }/4, \e^{-2}T]} \tilde{\xi} \geq \e^{b}.
\end{align*}
Hence, there exists interior maximum point $(x_0,t_0)$ of $\tilde{\xi}$ where
\begin{align}\label{l2e4}
\partial_t \tilde{\xi} \geq 0, \ \nabla \tilde{\xi} =0, \ \Delta \tilde{\xi} \leq 0, \ \mathrm{and}
\ \tilde{\xi} \geq \e^{b}
\end{align}
hold. By \eqref{l2e3} and \eqref{l2e4}, we obtain at the point $(x_0,t_0)$
\begin{align} \label{l2e5}
\partial_t \xi \geq -8\e^{2(1-\tilde{\beta})}M, \ |\nabla \xi| \leq 2s^{-1}_1\e M, \ 
\Delta \xi \leq 4n s_1^{-2}\e^2 M,\ \mathrm{and}
\ |\nabla \vhi|^2 \geq 2\e^{b}.
\end{align}
Substitute \eqref{l2e5} into \eqref{l2e2}. By $\e \nabla u \cdot (\nabla \vhi \otimes
\nabla \vhi)\leq 2\e |\nabla u|(\xi +W+G)$ and \eqref{l1e2}, we have
\begin{align}
0\leq& -(G')^2 -W'G'+2G''(\xi+W+G)
+\frac{4s_1^{-1}(|W'|+|G'|)\e M}{(2\e^{b})^{\frac12}} \nonumber\\
&+2\e^{1-\beta}(\xi+W+G)+8\e^{2(1-\tilde{\beta})}M +2s_1^{-1}\e^{2-\beta}M+4n s_1^{-2}\e^2 M . \label{l2e6}
\end{align}
Since $G''=-2\e^{b}$, $\xi+G\geq 0$, 
and $b\leq 1-\beta$, we have
\begin{align}
2G''(\xi+W+G)+2\e^{1-\beta}(\xi+W+G) \leq -2W\e^b. \label{l2e7}
\end{align}
By \eqref{l2e6} and \eqref{l2e7}, there exists a constant $c_{\cc \label{2c1}}= c_{\ref{2c1}}
(n,\beta,W,s)$ such that for sufficiently small $\e$ depending only on $n$, $\beta$, $W$,
$s$, and $\tau$, we have
\begin{align}
	\begin{split}
	0 \leq& -(G')^2 -W'G'-2W\e^b \\ &+c_{\ref{2c1}}M(|W'|\e^{1-\frac{b}{2}}
	+|G'|\e^{1-\frac{b}{2}}+ \e^{1-\frac{\beta}{2}}).
	\end{split} \label{l2e8}
\end{align}
If $|\vhi (x_0,t_0)|\leq \alpha$, then we have $W(\vhi(x_0,t_0))\geq \min_{|z|\leq \alpha} W(z)>0$. Since $W'G'\geq 0$ and $b\leq 1-\beta <1-\beta/2$,
 there exists a constant $c_{\cc \label{2cca1}}= c_{\ref{2cca1}}
 (n,\beta,W,s)$ such that
\begin{align} \label{2eea}
	0 \leq  -\e^b + c_{\ref{2cca1}} Me^{1-\frac{b}{2}}
\end{align}
for sufficiently small $\e$.
If $|\vhi (x_0,t_0)|>\alpha$, we have
\begin{align*}
-(G')^2 \leq -4\e^{2b}(\alpha-|\gamma|)^2,\\  -W'G'\leq
-2\e^{b}(\alpha-|\gamma|)|W'|,\\
|G'|\e^{1-\frac{b}{2}} \leq  2\e^{1+\frac{b}{2}}(\alpha-|\gamma|).
\end{align*}
By definition of $W$, we have $\alpha>|\gamma|$, and 
 there exists a constant $c_{\cc \label{2cca2}}= c_{\ref{2cca2}}
(n,\beta,W,s)$ such that
\begin{align} \label{2eeb}
	0 \leq  -|W'|(\e^b-c_{\ref{2cca2}}M \e^{1-\frac{b}{2}}) - (\e^{2b}- c_{\ref{2cca2}}M\e^{1-\frac{\beta}{2}})
\end{align}
for sufficiently small $\e$.
If $b= (4-3\beta)/8 $, we have $b<1-b/2$ and $2b<1-\beta/2$. \eqref{2eea} and \eqref{2eeb} are contradiction for sufficiently small $\e$. Hence, we obtain
\begin{align} \label{l2e9}
\sup_{\ U_{\e^{-1}s} \times [3\e^{-2({1-\tilde{\beta})} }/4, \e^{-2}T]} 
\left(\frac{|\nabla \vhi|^2}{2}-W(\vhi)\right) \leq 9\e^{\frac{4-3\beta}{8}},
\end{align}
where $G\leq 8\e^{b}$ is used. 

Now repeat the same argument, this time with $M$ replaced by $9\e^{b_k}$ and $G$ replaced
by $8\e^{b_{k+1}}(1-\frac18 (\vhi-\gamma)^2)$. To derive a contradiction, suppose that
\begin{align}\label{l2e10}
\sup_{\ U_{\e^{-1}s} \times [(1-(1/4)^k)\e^{-2(1-\tilde{\beta}) }, \e^{-2}T]} \xi \geq \e^{b_{k+1}}.
\end{align}
if $b_{k+1} \leq 1-\beta$, \eqref{2eea} is
 \begin{align*}
 0\leq -\e^{b_{k+1}}+ c_{\ref{2cca1}} e^{b_k+ 1-\frac{b_{k+1}}{2}}
 \end{align*}
and \eqref{2eeb} is
\begin{align*} 
	0 \leq  -|W'|(\e^{b_{k+1}}-c_{\ref{2cca2}}\e^{b_k+1-\frac{b_{k+1}}{2}}) - (\e^{2b_{k+1}}- c_{\ref{2cca2}}\e^{b_k+1-\frac{\beta}{2}}).
\end{align*}
If $b_{k+1} =\min \{b_k/2 +(4-3\beta)/8, 1-\beta \} $ and $b_k>0$, these equations are contradiction for sufficiently small $\e$. Hence, we obtain
\begin{align*}
	\sup_{\ U_{\e^{-1}s} \times [(1-(1/4)^k)\e^{-2(1-\tilde{\beta}) }, \e^{-2}T]} 
	\left(\frac{|\nabla \vhi|^2}{2}-W(\vhi)\right) \leq c_{\ref{2ccd1}}\e^{b_{k+1}},
\end{align*}
where $c_{\cc \label{2ccd1}}$ depends on $k$. Since the limit of $b_{k+1}=b_k/2 +(4-3\beta)/8$ is greater than $1-\beta$, there exists a constant $k'=k'(\beta)$ such that
$b_{k'+1} \geq 1-\beta >b_{k'}$, and this lemma follows.
\end{proof}
Define
\begin{align}
	&\eta(x) \in C^\infty_c(B_{\frac14}) \ \mathrm{with} \ \eta =1 \ \mathrm{on} \ 
	B_\frac18, \ 0\leq \eta \leq 1,\label{eta}\\
	&\tilde{\rho}_{(y,s)}(x,t) :=\rho_{(y,s)}(x,t) \eta(x-y)=
	\frac{1}{(4 \pi (s-t))^{\frac{n-1}{2}}} e^{-\frac{|x-y|^2}{4(s-t)}} \eta(x-y).
\end{align}
\begin{lemm} \label{l3} 
	Let $\iota$, $R$, $r$ be positive with $0 \leq \iota- \left(\frac{R}{r}\right)^2 \leq T$ and
	$R \in (0, \frac14)$. Set $\tilde{\iota} = \iota- \left(\frac{R}{r}\right)^2$. For any $y \in U_{\frac12}$,
	we have
	\begin{align}
	\int_{B_{\frac{1}{4}}(y)} \tilde{\rho}_{(y,\iota)}(x,\tilde{\iota})  \ d\mu_{\tilde{\iota}}\leq
	\left(\frac{r}{\sqrt{4\pi} R}\right)^{n-1} \left(\mu_{\tilde{\iota}}(B_{R}(y)) +
	\Lambda_1 e^{-\frac{r^2}{4}} \right).
	\end{align}
\end{lemm}
\begin{proof}
By \eqref{a5}, we compute
	\begin{align*}
	&\int_{B_{\frac{1}{4}}(y)} \tilde{\rho}_{(y,\iota)}(x,\tilde{\iota})  \ d\mu_{\tilde{\iota}} \\
	=& \int_{B_{R}(y)} \tilde{\rho}_{(y,\iota)}(x,\tilde{\iota})  \ d\mu_{\tilde{\iota}}
	+\int_{B_{\frac{1}{4}}(y) \backslash B_{R}(y)} \tilde{\rho}_{(y,\iota)}(x,\tilde{\iota})  \ d\mu_{\tilde{\iota}} \\
	\leq &\left(\frac{r}{\sqrt{4\pi} R}\right)^{n-1} \mu_{\tilde{\iota}}(B_{R}(y))+
	\left(\frac{r}{\sqrt{4\pi} R}\right)^{n-1} 
	\int_{B_{\frac{1}{4}}(y) \backslash B_{R}(y)} e^{-\frac{r^2|x-y|^2}{4R^2}} 
	 \ d\mu_{\tilde{\iota}}\\
	\leq &\left(\frac{r}{\sqrt{4\pi} R}\right)^{n-1} \left(\mu_{\tilde{\iota}}(B_{R}(y)) +
	\Lambda_1 e^{-\frac{r^2}{4}} \right).
	\end{align*}
\end{proof}
We can prove the following by exactly the same proof as \cite[Proposition 4.1]{TA16}.
\begin{theo} \label{l4} 
	There exists a constant $c_{\cc \label{4c1}} =c_{ \ref{4c1}}(n)$ such that
	\begin{align}
	&\frac{d}{dt}\int_{B_{\frac{1}{4}}(y)} \tilde{\rho}_{(y,s)}(x,t) \ d\mu_t(x) \nonumber \\  \leq&
	\frac12 \int_{B_{\frac{1}{4}}(y)} \tilde{\rho}_{(y,s)}(x,t) |u|^2  \ d\mu_t(x)
	\nonumber \\ &+  \frac{1}{2(s-t)} \int_{B_{\frac{1}{4}}(y)} \left(\frac{\e |\nabla \vhi|^2}{2} -\frac{W(\vhi)}{\e}\right)\tilde{\rho}_{(y,s)}(x,t) \ dx \nonumber \\
	&+ c_{\ref{4c1}}  e^{-\frac{1}{128(s-t)}} \mu_t(B_{\frac{1}{4}}(y)) \label{4e1}
	\end{align}
	for $y\in U_{\frac12}$, $0<t<s < \infty$ and $t<T$. 
\end{theo}
\begin{lemm} \label{l5}
	There exist constants  $c_{\cc \label{5c1}}=c_{\ref{5c1}}(n,W)$ and $\e_{\ce \label{5E1}}
	=\e_{\ref{5E1}}(n,W,\beta , \Lambda_1)$ with the following property. Assume $\e \leq \e_{\ref{5E1}}$ and there exist $y \in U_{\frac12}$ and $s \in (\e^2+2\e^{2\tilde{\beta}}, T]$ such that
	$|\vhi(y,s)| \leq \alpha$ holds. Then for any $y \in U_{\frac12}$ and $t \in (\e^2, T]$
	with $s-2\e^{2\tilde{\beta}} \leq t \leq s$ we have
	\begin{align}
	c_{\ref{5c1}} \leq \frac{|4n\log{\e}|^{\frac{n-1}{2}}}{R^{n-1}} \mu_t(B_R(y)),
	\end{align}
	where  $R = |4n\log{\e}|^{\frac12}\sqrt{s+\e^2-t} $.
\end{lemm}
\begin{proof}
Set $\tilde{\rho}(x,t) = \tilde{\rho}_{(y,s+\e^2)}(x,t)$ in this proof. We have
	\begin{align}
	\int_{B_{\frac{1}{4}}(y)} \tilde{\rho}(x,s)  \ d\mu_s(x) 
	= \int_{B_{\frac{1}{4 \e }}(0)} \frac{e^{-\frac{|\tilde{x}|^2}{4}}}{(\sqrt{4\pi})^{n-1}}
	\eta(\e \tilde{x}) \left( \frac{|\nabla \tilde{\vhi}|^2}{2} +W(\tilde{\vhi}) \right) \ d\tilde{x}
	\end{align}
	where $\tilde{x} =\frac{x-y}{\e}$ and $\tilde{\vhi}(\tilde{x},s) = \vhi (x,s)$.
	By $|\tilde{\vhi}(0,s)| \leq \alpha $, $\e^2 <s$, and Lemma \ref{l1},
	there exists a constant $0<c_{\cc \label{5c2}} = c_{\ref{5c2}}(n,W) <1$ such that
	\begin{align}
	2 c_{\ref{5c2}} \leq \int_{B_{\frac{1}{4}}(y)} \tilde{\rho}(x,s)  \ d\mu_s(x) . \label{5e1}
	\end{align}
	By \eqref{a3}, \eqref{a5}, Lemma \ref{l2}, Theorem \ref{l4}, and $\int_{B_1} \tilde{\rho} \  dx \leq \sqrt{4\pi (s-t)}$,
	if $\e< \e_{\ref{2E1}}$, then we obtain
	\begin{align}
	&\frac{d}{d\lambda}\int_{B_{\frac{1}{4}}(y)} \tilde{\rho}(x,\lambda) \ d\mu_\lambda(x) \nonumber \\
	\leq & \e^{-2\beta} \int_{B_{\frac{1}{4}}(y)} \tilde{\rho}(x,\lambda) \ d\mu_\lambda(x)
	+ \frac{\sqrt{\pi}\e^{-\beta} }{\sqrt{s-\lambda}} + c_{ \ref{4c1}} \Lambda_1 \label{5e2}
	\end{align}
	for $\lambda \in [\e^2,s)$. Multiply \eqref{5e2} by $e^{\e^{-2\beta}(s-\lambda)}$ and integrate over $[t,s]$.
	By $s-2\e^{2\tilde{\beta}} \leq t \leq s$, we have
	\begin{align}
	&\int_{B_{\frac{1}{4}}(y)} \tilde{\rho}(x,s)  \ d\mu_s(x)  -
	e^{\e^{-2\beta}(s-t)}\int_{B_{\frac{1}{4}}(y)} \tilde{\rho}(x,t)  \ d\mu_t(x) \nonumber\\
	\leq& 2\sqrt{2\pi} \e^{\tilde{\beta} -\beta} e^{2\e^{2(\tilde{\beta}-\beta)}}
	+ 2c_{ \ref{4c1}} \Lambda_1 \e^{2\beta} e^{2\e^{2(\tilde{\beta}-\beta)}}.\label{5e3}
	\end{align}
	By \eqref{5e1} and \eqref{5e3}, for sufficiently small $\e$ depending only on
	$n$, $W$, $\beta$, and $\Lambda_1$, we have
	\begin{align}
	c_{\ref{5c2}} \leq
	\int_{B_{\frac{1}{4}}(y)} \tilde{\rho}(x,t)  \ d\mu_t(x) . \label{5e4}
	\end{align}
	Next, set $\iota = s+\e^2$, $r= |4n\log{\e}|^{\frac12}$, and  $R= |4n\log{\e}|^{\frac12}\sqrt{s+\e^2-t}$.
	Then we have
	\begin{align}
	\iota - \left(\frac{R}{r}\right)^2  =(s+\e^2) -\left( \frac{ |4n\log{\e}|^{\frac12}\sqrt{s+\e^2-t}}{ |4n\log{\e}|^{\frac12}}\right)^{2} =t \leq T. \label{5e5}
	\end{align}
	For sufficiently small $\e$ depending on $n$ and $\beta$, we have
	\begin{align}
	R= |4n\log{\e}|^{\frac12}\sqrt{s+\e^2-t} \leq  |4n\log{\e}|^{\frac12}\sqrt{\e^2 +2\e^{2\tilde{\beta}}}<\frac14. \label{5e6}
	\end{align}
	By \eqref{5e5}, \eqref{5e6}, and Lemma \ref{l3}, we obtain
	\begin{align}
	\int_{B_{\frac{1}{4}}(y)} \tilde{\rho}(x,t)  \ d\mu_{t}&\leq
	\left(\frac{r}{\sqrt{4\pi} R}\right)^{n-1} \left(\mu_{t}(B_{R}(y)) +
	\Lambda_1 e^{-\frac{r^2}{4}} \right) \nonumber \\
	&\leq 	\left(\frac{r}{\sqrt{4\pi} R}\right)^{n-1}\mu_{t}(B_{R}(y))
	+ \frac{\Lambda_1}{(4\pi)^{\frac{n-1}{2}} }\e. \label{5e7}
	\end{align}
	By \eqref{5e4} and \eqref{5e7}, 
	for sufficiently small $\e$ depending on $n$, $W$, $\beta$, and $\Lambda_1$, we obtain
	\begin{align*}
	\frac{c_{\ref{5c2}}(\sqrt{4\pi})^{n-1}}{2 } \leq \left(\frac{r}{R}\right)^{n-1}\mu_{t}(B_{R}(y)).
	\end{align*}
	Thus this lemma follows.
\end{proof}
Define
\begin{align*}
E(t):= \sup_{y \in B_{\frac14}, r \leq \frac14} \frac{\mu_t(B_r(y))}{r^{n-1}}.
\end{align*}
Note that $\sup_{t \in [\e^2, T]} E(t)$ is bounded for each $\e$ by Lemma \ref{l1}.
\begin{lemm} \label{l6} 
	There exist constants $c_{\cc \label{6c1}}=c_{\ref{6c1}}(n,W,\beta)$ and $\e_{\ce \label{6E1}} =\e_{\ref{6E1}}(n,W,\beta,\Lambda_1)$ with following property.
	if $\e < \e_{\ref{6E1}}$, for any $y \in B_{\frac14}$, $r \in (4|n\log{\e}|^{\frac12}\e^{\tilde{\beta}}, \frac{1}{4}]$  and $t \in (\e^2+3\e^{2\tilde{\beta}} ,T]$, we have
\begin{align}
&\int_{B_r(y)} \left(\frac{\e |\nabla \vhi (x,t)|^2}{2} -\frac{W(\vhi(x,t))}{\e}\right)_+ \ dx \nonumber \\
\leq& c_{\ref{6c1}} \e^{\tilde{\beta} -\beta} |n\log{\e}|^{\frac{n}{2}} r^{n-1} \left(1+\sup_{\lambda \in [t-2\e^{2\tilde{\beta}}, t]}E(\lambda)\right). \label{6e1}
\end{align}
\end{lemm}
\begin{proof}
	Fix $y \in B_{\frac14}$, $r \in (4|n\log{\e}|^{\frac12}\e^{\tilde{\beta}}, \frac{1}{4}]$, and $t_* \in (\e^2+3\e^{2\tilde{\beta}} ,T]$. We estimate the integral of \eqref{6e1} by separating $B_r(y)$
	into three disjoint sets $\mathcal{A}$, $\mathcal{B}$, and $\mathcal{C}$.
\begin{align*}
&\mathcal{A} := B_r(y) \backslash B_{r- 4|n\log{\e}|^{\frac12}\e^{\tilde{\beta}}}(y),\\
&\tilde{\mathcal{B}} := \{x \in B_{r- 4|n\log{\e}|^{\frac12}\e^{\tilde{\beta}}}(y): \mathrm{for \ some \ } \tilde{t} \ \mathrm{with} \ 
t_*- \e^{2\tilde{\beta}} \leq \tilde{t} \leq t_*, |\vhi(x,\tilde{t})| \leq \alpha \},\\
&\mathcal{B}:=\{x\in B_{r- 4|n\log{\e}|^{\frac12}\e^{\tilde{\beta}}}(y): \dist(\tilde{\mathcal{B}},x)<4|n\log{\e}|^{\frac12}\e^{\tilde{\beta}} \},\\
&\mathcal{C} := \{x\in B_{r- 4|n\log{\e}|^{\frac12}\e^{\tilde{\beta}}}(y): \dist(\tilde{\mathcal{B}},x)\geq 4|n\log{\e}|^{\frac12}\e^{\tilde{\beta}} \}.
\end{align*}
We first estimate on $\mathcal{A}$. Using Lemma \ref{l2}, we compute
\begin{align}
&\int_{B_r(y) \cap \mathcal{L}^n (\mathcal{A}) } \left(\frac{\e |\nabla \vhi (x,t_*)|^2}{2} -\frac{W(\vhi(x,t_*))}{\e}\right)_+ \ dx \nonumber \\
\leq& \e^{-\beta} \mathcal{L}^n (\mathcal{A}) \nonumber\\
\leq&4n \omega_n r^{n-1}\e^{\tilde{\beta} -\beta} |n\log{\e}|^{\frac12}.\label{6e1a}
\end{align}

We next estimate on $\mathcal{B}$. Using Vitali's covering theorem, there exists a disjoint family of balls
$\{B_{4|n\log{\e}|^{\frac12}\e^{\tilde{\beta}}}(x_i) \}^N_{i=1}
\subset \{B_{4|n\log{\e}|^{\frac12}\e^{\tilde{\beta}} }(x) : x \in \tilde{\mathcal{B}}\}$ such that
\begin{align}
x_i \in \tilde{\mathcal{B}}\subset U_{\frac12} \ \mathrm{for \ each} \ i=1,\ldots,N \  \mathrm{and} \ 
\mathcal{B}\subset \cup^N_{i=1} B_{20|n\log{\e}|^{\frac12}\e^{\tilde{\beta}} }(x_i), \label{6e2} \\
\e^2+2\e^{2\tilde{\beta}} <t_* - \e^{2\tilde{\beta}} \leq \tilde{t}_i \leq t_*\leq T, \ |\vhi(x_i,\tilde{t}_i)| \leq \alpha. \label{6e3}
\end{align}
Note that $\tilde{t}_i- 2\e^{2\tilde{\beta}} \leq t_* -2\e^{2\tilde{\beta}} 
\leq \tilde{t}_i$.
Then the assumption of Lemma \ref{l5} is satisfied for $s=\tilde{t}_i$, $y=x_i$, $t= t_*-2\e^{2\tilde{\beta}}$, and $R_i :=  |4n\log{\e}|^{\frac12}\sqrt{\tilde{t}_i+\e^2-t_*+2\e^{2\tilde{\beta}}}$ if $\e \leq \e_{\ref{5E1}}$. Thus we obtain
\begin{align}
c_{\ref{5c1}}R_i^{n-1} \leq |4n\log{\e}|^{\frac{n-1}{2}} \mu_{t_*-2\e^{2\tilde{\beta}}}(B_{R_i}(x_i)) \quad \mathrm{for } \ i=1,\ldots ,N. \label{6e4}
\end{align}
By \eqref{6e3} and definition of $R_i$, we have 
\begin{align}
|4n\log{\e}|^{\frac{1}{2}} \sqrt{\e^2+\e^{2\tilde{\beta}}} \leq R_i \leq |4n\log{\e}|^{\frac{1}{2}} \sqrt{ \e^2+2\e^{2\tilde{\beta}}} \leq 2 \e^{\tilde{\beta}} |4n\log{\e}|^{\frac{1}{2}}. \label{6e5}\end{align}
By \eqref{6e4} and \eqref{6e5}, we obtain
\begin{align}
c_{\ref{5c1}} \e^{\tilde{\beta}(n-1)} \leq \mu_{t_*-2\e^{2\tilde{\beta}}}(B_{4|n\log{\e}|^{\frac12}\e^{\tilde{\beta}}}(x_i)) \quad \mathrm{for } \ i=1,\ldots ,N. \label{6e6}
\end{align}
Since $\{B_{4|n\log{\e}|^{\frac12}\e^{\tilde{\beta}}}(x_i) \}^N_{i=1}$ are pairwise disjoint balls
 and $B_{4|n\log{\e}|^{\frac12}\e^{\tilde{\beta}}}(x_i)  \subset B_{r}(y)$, by summing \eqref{6e6} over $i$,
\begin{align}
N c_{\ref{5c1}} \e^{\tilde{\beta}(n-1)} \leq \mu_{t_*-2\e^{2\tilde{\beta}}}(B_{r}(y)). \label{6e7}
\end{align}
Hence, by \eqref{6e2}, \eqref{6e7}, and Lemma \ref{l2}, we obtain
\begin{align}
&\int_{B_r(y) \cap \mathcal{L}^n (\mathcal{B}) } \left(\frac{\e |\nabla \vhi (x,t_*)|^2}{2} -\frac{W(\vhi(x,t_*))}{\e}\right)_+ \ dx \nonumber \\
\leq&\e^{-\beta} \mathcal{L}^n (\mathcal{B}) \nonumber\\
\leq&\e^{-\beta} N (20|n\log{\e}|^{\frac12}\e^{\tilde{\beta}})^n \nonumber \\
\leq&\frac{20^n}{c_{\ref{5c1}} } \e^{\tilde{\beta}-\beta} |n\log{\e}|^{\frac{n}{2}}
\mu_{t_*-2\e^{2\tilde{\beta}}}(B_{r}(y)). \label{6e8a}
\end{align}

We finally estimate on $\mathcal{C}$. Define
\begin{align*}
\psi(z) := \min\{1, (4|n\log{\e}|^{\frac12}\e^{\tilde{\beta}} )^{-1} \dist (z, \{x:|x-y|\geq r\} \cup \tilde{\mathcal{B}} ) \}.
\end{align*}
$\psi$ is a Lipschitz function and is $0$ on  $ \{x:|x-y|\geq r\} \cup\tilde{\mathcal{B}} $, $1$ on $\mathcal{C}$, and $|\nabla \psi | \leq   (4|n\log{\e}|^{\frac12}\e^{\tilde{\beta}})^{-1}$.
Differentiate \eqref{MTe2} with respect to $x_j$, multiply it by $\vhi_{x_j} \psi^2$ and sum over $j$,
\begin{align}
&\frac{d}{dt} \int_{B_r(y)} \frac{1}{2} |\nabla \vhi |^2 \psi^2 \ dx \nonumber\\ =&
\int_{B_r(y)} (-u \otimes \nabla \vhi \cdot \nabla^2 \vhi - \nabla \vhi \otimes \nabla \vhi
\cdot \nabla u + \nabla \vhi \cdot \Delta \nabla \vhi -\frac{W''(\vhi)}{\e^2} |\nabla \vhi|^2)\psi^2 \ dx. \label{6e8}
\end{align}
 By \eqref{6e8}, integration by part, and Young's inequality, we obtain
\begin{align}
&\frac{d}{dt} \int_{B_r(y)} \frac{1}{2} |\nabla \vhi |^2 \psi^2 \ dx \nonumber\\
\leq& \int_{B_r(y)} |u|^2 |\nabla \vhi|^2 \psi^2 +\frac{ |\nabla^2 \vhi|^2\psi^2 }{4}+ |\nabla \vhi|^2
|\nabla u| \psi^2 - |\nabla^2 \vhi|^2 \psi^2 \nonumber\\
& + \frac{ |\nabla^2 \vhi|^2\psi^2 }{4} + |\nabla \vhi|^2 |\nabla \psi|^2  - 
\frac{W''(\vhi)}{\e^2} |\nabla \vhi|^2\psi^2 \ dx \nonumber \\
\leq & \int_{B_r(y)} (|u|^2 +|\nabla u| - 
\frac{W''(\vhi)}{\e^2} ) |\nabla \vhi|^2 \psi^2+ |\nabla \vhi|^2 |\nabla \psi|^2 \ dx  . \label{6e9}
\end{align}
Since $|\vhi| \geq \alpha$ on the support on $\psi$, we have $W''(\vhi) \geq \kappa$.
Thus, by \eqref{a3} and $|\nabla \psi | \leq   (4|n\log{\e}|^{\frac12}\e^{\tilde{\beta}})^{-1}$, \eqref{6e9} gives
\begin{align}
\frac{d}{dt} \int_{B_r(y)} \frac{1}{2} |\nabla \vhi |^2 \psi^2 \ dx
&\leq -\frac{\kappa}{2\e^2} \int_{B_r(y)} |\nabla \vhi|^2 \psi^2 \ dx
+\int_{B_r(y)} (4|n\log{\e}|^{\frac12}\e^{\tilde{\beta}})^{-2} |\nabla \vhi |^2 \ dx \label{e610}
\end{align}
for sufficiently small $\e$ depending on $\beta$ and $W$.
Multiply \eqref{e610} by $e^{-\kappa \e^{-2} (t_* -\lambda)} $ and integrate  over $[t_*- \e^{2\tilde{\beta}},t_*]$, we have
\begin{align}
& \int_{B_r(y)} \frac{1}{2} |\nabla \vhi (x,t_*) |^2 \psi^2  (x,t_*) \ dx \nonumber\\
\leq& e^{-\kappa \e^{2(\tilde{\beta}-1)}}  
\int_{B_r(y)} \frac{1}{2} |\nabla \vhi (x,t_*- \e^{2\tilde{\beta}}) |^2 \psi^2  (x,t_*- \e^{2\tilde{\beta}}) \ dx \nonumber\\
&+ \int^{t_*}_{t_*- \e^{2\tilde{\beta}}} e^{-\kappa \e^{-2} (t_* -\lambda)} 
 (4|n\log{\e}|^{\frac12}\e^{\tilde{\beta}})^{-2} \left( \int_{B_r(y)} |\nabla \vhi (x,\lambda)|^2\ dx  \right) \ d \lambda \nonumber\\
\leq & \left( e^{-\kappa \e^{2(\tilde{\beta}-1)}} + 2 \kappa^{-1}\e^2
 (4|n\log{\e}|^{\frac12}\e^{\tilde{\beta}})^{-2}  \right) \e^{-1} \sup_{\lambda_* \in [t_*-2\e^{2\tilde{\beta}}, t_*]} \mu_{\lambda_*}(B_{r}(y)).
\end{align}
Since $\tilde{\beta}-\beta <2-2\tilde{\beta}$,
there exists a constant $c_{\cc \label{6c2}} = c_{\ref{6c2}}(n,W,\beta)$ such that
\begin{align}
 \int_{B_r(y)} \frac{\e}{2} |\nabla \vhi (x,t_*) |^2 \psi^2  (x,t_*) \ dx
 \leq c_{\ref{6c2}} \e^{\tilde{\beta}-\beta }
\sup_{\lambda_* \in [t_*-2\e^{2\tilde{\beta}}, t_*]} \mu_{\lambda_*}(B_{r}(y))
\end{align}
for sufficiently small $\e$ depending on $n$, $\beta$, and $W$. Then we obtain
\begin{align}
&\int_{B_r(y) \cap \mathcal{L}^n (\mathcal{C}) } \left(\frac{\e |\nabla \vhi (x,t_*)|^2}{2} -\frac{W(\vhi(x,t_*))}{\e}\right)_+ \ dx \nonumber \\
\leq&\int_{B_r(y)} \frac{\e}{2} |\nabla \vhi (x,t_*) |^2 \psi^2  (x,t_*) \ dx \nonumber\\
 \leq& c_{\ref{6c2}} \e^{\tilde{\beta}-\beta }
\sup_{\lambda_* \in [t_*-2\e^{2\tilde{\beta}}, t_*]} \mu_{\lambda_*}(B_{r}(y)).\label{6e10a}
\end{align}
This lemma follows by \eqref{6e1a}, \eqref{6e8a}, and \eqref{6e10a}.
\end{proof}
\begin{lemm}\label{l7} 
There exists a constant $c_{\cc \label{7c1}} = c_{\ref{7c1}}(n,W, \beta,T)$ such that, if $\e < \e_{\ref{6E1}}$, for any $y \in B_{\frac14}$,
$t \in [4\e^{2\tilde{\beta}},T]$, and $t<s$, we have
\begin{align}
&\int^t_{2\e^{2\tilde{\beta}}} \frac{1}{2(s-\lambda)} \int_{B_{\frac{1}{4}}(y)} \left(\frac{\e |\nabla \vhi|^2}{2} -\frac{W(\vhi)}{\e}\right)_+ \tilde{\rho}_{(y,s)}(x,\lambda) \ dxd\lambda \nonumber\\
\leq& c_{\ref{7c1}} \e^{\tilde{\beta}-\beta}|n\log{\e}|^{\frac{n+2}{2}}\left(1+\sup_{\lambda _*\in [2\e^{2\tilde{\beta}}, t]}E(\lambda_*)\right) .
\end{align}
\end{lemm}
\begin{proof}
	If $s-2\e^{2\tilde{\beta}} \leq t < s$, by Lemma \ref{l2} and $\int \rho \ dx = \sqrt{4\pi (s-\lambda)}$, we have
	\begin{align}
	&\int^t_{s-2\e^{2\tilde{\beta}}} \frac{1}{2(s-\lambda)} \int_{B_{\frac{1}{4}}(y)} \left(\frac{\e |\nabla \vhi|^2}{2} -\frac{W(\vhi)}{\e}\right)_+ \tilde{\rho}_{(y,s)}(x,\lambda) \ dxd\lambda \nonumber\\
	\leq & \sqrt{\pi}\int^t_{s-2\e^{2\tilde{\beta}}} \frac{\e^{-\beta}}{\sqrt{s-\lambda}} \ d \lambda
	\leq 2 \sqrt{2\pi} \e^{\tilde{\beta}-\beta}.\label{7e1}
	\end{align}
	Hence we only need to consider $t \in  [4\e^{2\tilde{\beta}},s-2\e^{2\tilde{\beta}}]$.
	Using Lemma \ref{l6}, we first estimate on $B_{8|n\log{\e}|^{\frac12}\e^{\tilde{\beta}}}(y)$.
	\begin{align}
	&\int^t_{4\e^{2\tilde{\beta}}} \frac{1}{2(s-\lambda)} \int_{B_{8|n\log{\e}|^{\frac12}\e^{\tilde{\beta}}}(y)} \left(\frac{\e |\nabla \vhi|^2}{2} -\frac{W(\vhi)}{\e}\right)_+ \tilde{\rho}_{(y,s)}(x,\lambda) \ dxd\lambda \nonumber\\
	\leq&\int^{t}_{4\e^{2\tilde{\beta}}} \frac{8^n\omega_n c_{\ref{6c1}}}{2(\sqrt{4\pi})^{n-1}} \e^{n\tilde{\beta}-\beta} 
	|n\log{\e}|^{\frac{2n-1}{2}} (s-\lambda)^{-\frac{n+1}{2}}\left(1+\sup_{\lambda_* \in [\lambda-2\e^{2\tilde{\beta}}, \lambda]}E(\lambda_*)\right) \ d\lambda \nonumber \\
		\leq & c_{\ref{7c2}} \e^{\tilde{\beta}-\beta}|n\log{\e}|^{\frac{2n-1}{2}}\left(1+\sup_{\lambda_* \in [2\e^{2\tilde{\beta}}, t]}E(\lambda_*)\right) , \label{7e2}
	\end{align}
	where $c_{\cc \label{7c2}}$ depends only on $n$, $W$, and $\beta$. 
	
	We next estimate on $B_{\frac14}(y) \backslash B_{8|n\log{\e}|^{\frac12}\e^{\tilde{\beta}}}(y)$. Set $\xi_\e :=  \left(\frac{\e |\nabla \vhi|^2}{2} -\frac{W(\vhi)}{\e}\right)_+$. Using Lemma \ref{l6}, we compute
	\begin{align}
	&\int^t_{4\e^{2\tilde{\beta}}} \frac{1}{2(s-\lambda)} \int_{B_{\frac14}(y) \backslash B_{8|n\log{\e}|^{\frac12}\e^{\tilde{\beta}}}(y)} \left(\frac{\e |\nabla \vhi|^2}{2} -\frac{W(\vhi)}{\e}\right)_+ \tilde{\rho}_{(y,s)}(x,\lambda) \ dxd\lambda \nonumber\\
	\leq& \int^t_{4\e^{2\tilde{\beta}}} \frac{1}{2(\sqrt{4\pi})^{n-1} (s-\lambda)^{\frac{n+1}{2}}} \ d\lambda
	\int^1_0 dl \int _{B_{\frac{1}{4}}(y) \cap \{x ; e^{-\frac{|x-y|^2}{4(s-\lambda)}} \geq l\} \backslash B_{8|n\log{\e}|^{\frac12}\e^{\tilde{\beta}}}(y)} \xi_\e \ dx \nonumber \\
	\leq& \int^t_{4\e^{2\tilde{\beta}}}\frac{1}{2(\sqrt{4\pi})^{n-1} (s-\lambda)^{\frac{n+1}{2}}} d\lambda \int_{e^{-\frac{1}{64(s-\lambda)}}}^{e^{-\frac{16|n\log{\e}|\e^{2\tilde{\beta}}}{(s-\lambda)}}} dl
	\int_{B_{2\sqrt{ (s-\lambda) \log{l^{-1}}}}(y)} \xi_\e \ dx \nonumber\\
	\leq& c_{\ref{7c3}}\e^{\tilde{\beta} -\beta} |n\log{\e}|^{\frac{n}{2}}
	\left(1+\sup_{\lambda_* \in [2\e^{2\tilde{\beta}}, t]}E(\lambda_*)\right) 
	\nonumber\\
	& \int^t_{4\e^{2\tilde{\beta}}} (s-\lambda)^{-1} \int_{e^{-\frac{1}{64(s-\lambda)}}}^{e^{-\frac{16|n\log{\e}|\e^{2\tilde{\beta}}}{(s-\lambda)}}} (\log{l^{-1}})^{\frac{n-1}{2}} dld\lambda, \label{7e3}
	\end{align}
	where $c_{\cc \label{7c3}}$ depends only on $n$, $W$, and $\beta$. 
	By the change of variable, $t \leq s-2\e^{2\tilde{\beta}}$, and $t \leq T $, we have
	\begin{align}
	&\int^t_{4\e^{2\tilde{\beta}}} (s-\lambda)^{-1} \int_{e^{-\frac{1}{64(s-\lambda)}}}^{e^{-\frac{16|n\log{\e}|\e^{2\tilde{\beta}}}{(s-\lambda)}}} (\log{l^{-1}})^{\frac{n-1}{2}} dl d \lambda
	\nonumber\\
	\leq&\int^t_{4\e^{2\tilde{\beta}}} (s-\lambda)^{-1} \int^{\frac{1}{64(s-\lambda)}}_{\frac{16|n\log{\e}|\e^{2\tilde{\beta}}}{(s-\lambda)}}
	\tilde{l}^{\frac{n-1}{2}} e^{-\tilde{l}} \ d\tilde{l}d \lambda \nonumber\\
	\leq&c_{\ref{7c4}} \int^t_{4\e^{2\tilde{\beta}}} (s-\lambda)^{-1} 
	\int^{\frac{1}{64(s-\lambda)}}_{\frac{16|n\log{\e}|\e^{2\tilde{\beta}}}{(s-\lambda)}}
	 e^{-\frac{\tilde{l}}{2}} \ d\tilde{l}d \lambda \nonumber\\
	 \leq& 2 c_{\ref{7c4}} \log{\left(\frac{s-4\e^{2\tilde{\beta}}+(t-t)}{s-t}\right)} \nonumber\\
	  \leq & 2 c_{\ref{7c4}} \log{\left(\frac{T}{2\e^{2\tilde{\beta}}}\right)}, \label{7e4}
	\end{align}
	where $c_{\cc \label{7c4}}$ depends only on $n$. By \eqref{7e3} and \eqref{7e4},
	there exists a constant $c_{\cc \label{7c5}} = c_{\ref{7c5}}(n,W,\beta,T)$ such that
	\begin{align}
	&\int^t_{4\e^{2\tilde{\beta}}} \frac{1}{2(s-\lambda)} \int_{B_{\frac14}(y) \backslash B_{8|n\log{\e}|^{\frac12}\e^{\tilde{\beta}}}(y)} \left(\frac{\e |\nabla \vhi|^2}{2} -\frac{W(\vhi)}{\e}\right)_+ \tilde{\rho}_{(y,s)}(x,\lambda) \ dxd\lambda \nonumber\\
	\leq& c_{\ref{7c5}}\e^{\tilde{\beta} -\beta} |n\log{\e}|^{\frac{n+2}{2}}
	\left(1+\sup_{\lambda_* \in [2\e^{2\tilde{\beta}}, t]}E(\lambda_*)\right) \label{7e5}
	\end{align}
	This lemma follows by \eqref{7e1}, \eqref{7e2}, and \eqref{7e5}.
\end{proof}
To proceed, we need the following theorem (see \cite[Theorem 5.12.4]{WP72}).
\begin{theo}
	Let $\mu$ be a positive Radon measure on $\mathbb{R}^n$ satisfying
	\[
	K(\mu):=\sup_{B_r(x)\subset\R^n} \frac{1}{r^{n-1}} \mu(B_r(x))< \infty.
	\]
	Then there exists a constant $c(n)$ such that
	\[
	\left|\int_{\mathbb{R}^n} \phi  \,d\mu \right| \leq c(n) K(\mu) \int_{\mathbb{R}^n} |\nabla \phi| \,d\mathcal{L}^n
	\]
	for all $\phi \in C^1_c(\mathbb{R}^n)$. \label{4MZ}
\end{theo} 
\begin{lemm} \label{ll9}
	Given $s \in(0,1)$ and $s' \in (0,1-s)$. Let $\mu$ be a positive Radon measure on $U_1$
	satisfying
	\begin{align*}
	K(\mu,s,s'):= \sup_{x \in B_s,r\leq s'} \frac{1}{r^{n-1}} \mu(B_r(x))< \infty.
	\end{align*}
	Then there exists a constant $c_{\cc \label{9cc1}}=c_{\ref{9cc1}}(n,s,s')$ such that
	\begin{align*}
	\sup_{B_r(x)\subset \R^n}  \frac{1}{r^{n-1}} \mu(B_s \cap B_r(x)) \leq c_{\ref{9cc1}}
	K(\mu,s,s').
	\end{align*}
\end{lemm}
\begin{proof}
	For any $B_r(x)\subset \R^n$, if $r \leq s'$, there exists a set $B_r(y) \subset \R^n$ such that
	$y \in B_s$ and $B_s \cap B_r(x) \subset B_s \cap B_r(y)$. Hence, we have
	\begin{align} \label{ll9e1}
	\frac{1}{r^{n-1}} \mu(B_s \cap B_r(x)) \leq\frac{1}{r^{n-1}} \mu(B_s \cap B_r(y)) 
	\leq K(\mu,s,s').
	\end{align}
	Next, we consider the case of $r>s'$.
	 Let $\{B_{s'} (y_i)\}_{i=1}^{m}$ be sets such that $y_i \in B_s$ and $B_s \subset 
	\cup_{i=1}^m  B_{s'} (y_i)$. Here, $m$ depends only on $n$, $s$, and $s'$.
	Then we obtain
	\begin{align}\label{ll9e2}
	\frac{1}{r^{n-1}} \mu(B_s \cap B_r(x))
	\leq \sum_{i=1}^{m} \frac{1}{r^{n-1}} \mu(B_s'(y_i)) \leq m K(\mu,s,s').
	\end{align}
	This lemma follows by \eqref{ll9e1} and \eqref{ll9e2}.
\end{proof}
\begin{lemm}\label{l8} 
	There exists a constant $c_{\cc \label{8c1}}=c_{\ref{8c1}}(n,p,q)$ such that for any
		$t_0$, $t_1$ with $0\leq t_0<t_1<s$, we have
		\begin{align}
		\int^{t_1}_{t_0} \int_{B_{\frac14}} \tilde{\rho}_{(0,s)}|u|^2 \ d\mu dt \leq 
		c_{\ref{8c1}} (t_1-t_0)^{\hat{p}} \|u\|^2_{L^q([t_0,t_1];(W^{1,p}(\Omega))^n)} \sup_{t \in [t_0,t_1]} E(t),
		\end{align}
	where $(1)$ $0<\hat{p}=\frac{2pq-2p-nq}{pq}$ when $p<n$, $(2)$ $\hat{p}<\frac{q-2}{q}$
	may be taken arbitrarily close to $\frac{q-2}{q}$ when $p=n$ (and $c_{\ref{8c1}}$ depends
	on $\hat{p}$), and $(3)$ $\hat{p}<=\frac{q-2}{q}$ when $p>n$.
\end{lemm}
\begin{proof}
	By $\mu(B_{\frac14}) \leq E(t)$ and Lemma \ref{ll9}, we can prove this lemma as well as
	\cite[Lemma 4.8]{TA16}.
\end{proof}
\begin{theo}\label{l9}
	Set $l(x,t) =  \min \{\dist(x, \partial U_{1-\e}), \sqrt{t-\e^2} \}$. Then, there exists a constant $0<c_{\cc \label{9c1}}=c_{\ref{9c1}}(n,p,q,\beta,W,\Lambda_0, \Lambda_1)$ such that, if $0<\e<1/2$, $t > \e^2$ and $U_{2r}(x)\subset U_{1-\e}$, we have
\begin{align}
\frac{l^{n-1}(x,t)}{r^{n-1}} \mu_{t} (B_r(x))	\leq c_{\ref{9c1}}.	
\end{align}
\end{theo}
\begin{proof}
	Define
	\begin{align}
	E_1 := \sup_{U_{2r}(x)\subset U_{1-\e}, t>\e^2}  \frac{l^{n-1}(x,t)}{r^{n-1}} \mu_{t} (B_r(x)).	
	\end{align}
	Note that $E_1$ is bounded for each $\e$ by Lemma \ref{l1}.
	Let $x_0$, $r_0$, and $t_0$ be fixed such that
	\begin{align}
	\frac{l^{n-1}(x_0,t_0)}{r_0^{n-1}} \mu_{t_0} (B_{r_0}(x_0)) > \frac34 E_1.	\label{9e1}
	\end{align}
	Set $l=l(x_0,t_0)$, $\tilde{x} = (x-x_0) / l$, $\tilde{r}=r/l$, $\tilde{t}= (t-t_0)/l^2 +1$, $\tilde{\e} = \e/l$, $\tilde{\vhi}(\tilde{x},\tilde{t}) = \vhi(x,t)$ and $\tilde{u}(\tilde{x},\tilde{t}) = lu(x,t)$. 
	We consider on $\tilde{t} \in [0,1]$. 	Note that $U_{l+\e}(x_0)\subset U_1$. We have
	\begin{align}
	\partial_t \tilde{\vhi} + \tilde{u}\cdot \nabla \tilde{\vhi} = \Delta \tilde{\vhi}- \frac{W'(\tilde{\vhi} )}{\tilde{\e}^2 } \quad \mathrm{for } \ (\tilde{x}, \tilde{t})
	\in U_{1+\tilde{\e}}\times [0,1], \label{9e2} \\
	\|\tilde{u}\|_{L^q([0,1];(W^{1,p}(U_{1+\tilde{\e}}))^n)} \leq l^{2-\frac{n}{p}-\frac{2}{q}}  \Lambda_0, \label{9e2a}\\
	\int_{U_{1+\tilde{\e}}} \left(\frac{\tilde{\e} |\nabla \tilde{\vhi}|^2 }{2}
	+ \frac{W(\tilde{\vhi})}{\tilde{\e}} \right) \ d \tilde{x} \leq l^{1-n}  \Lambda_1, \label{9e3}\\
	\frac{\tilde{\mu}_{\tilde{t}}(B_{\tilde{r}}(\tilde{x})) }{\tilde{r}^{n-1}}	=
	\frac{1}{\tilde{r}^{n-1}}	\int_{B_{\tilde{r}}(\tilde{x})} \left(\frac{\tilde{\e} |\nabla \tilde{\vhi}|^2 }{2}
	+ \frac{W(\tilde{\vhi})}{\tilde{\e}} \right)\ d \tilde{x}  = 
	\frac{\mu_t(B_r(x))}{r^{n-1}} , \label{9e4}
	\end{align}
	for $B_{\tilde{r}}(\tilde{x})  \subset U_{1+\tilde{\e}}$.
	If $\tilde{r}_0 = \frac{r_0}{l} \geq \frac14$, we have
	\begin{align}
		\frac{l^{n-1}}{r_0^{n-1}} \mu_{t_0} (B_{r_0}(x_0))	 \leq 4^{n-1} \Lambda_1. \label{9e5}
	\end{align}
	By \eqref{9e1} and \eqref{9e5}, we obtain
	\begin{align}
	E_1 \leq \frac{4^n}{3} \Lambda_1.\label{9e1f}
	\end{align}
	We only need to consider $\tilde{r}_0 \leq \frac14$.
	For any $x\in B_{l/4}(x_0)$, we have ${\rm dist}\,(x,\partial U_{1-\e})\geq3 l/4 > l/2$. Hence
	for any $x\in B_{l/4}(x_0)$ and $(t_0+3\e^2)/4<t<t_0$, we have 
	\begin{align*}
	\frac{l^{n-1}(x,t)}{r^{n-1}} \mu_{t} (B_r(x)) \leq E_1
	\end{align*}
by the definition of $E_1$ and $r<l/4\leq {\rm dist}\,(x,\partial U_{1-\e})/2$.
	Using ${\rm dist}(x,\partial U_{1-\e})\geq l/2$ and $\sqrt{t-\e^2}\geq l/2$, this gives
	\begin{align}
	\tilde{E} := \sup_{\tilde{x}\in B_{\frac{1}{4}},\, 0<\tilde{r}<\frac{1}{4}, \frac12<\tilde{t}\leq 1}\frac{\tilde{\mu}_{\tilde{t}}(B_{\tilde{r}}(\tilde{x})) }{\tilde{r}^{n-1}} \leq 2^{n-1}l^{1-n}E_1. \label{9e6}
	\end{align}
	Set
	\begin{align}
	\hat{\e}:= \min \{\e_{\ref{2E1}}, \e_{\ref{5E1}}, \e_{\ref{6E1}},\check{\e}, \frac12 \}, \label{9e6a} \\
	\hat{t} := \min \{\frac14,(2^n(\sqrt{4\pi})^{n-1} e^{\frac{1}{16}}  c_{\ref{8c1}} \Lambda_0^2)^{-\frac{1}{\hat{p}}} \}. \label{9e6b}
	\end{align}
	Here, $\check{\e}$ satisfies that if $0<\delta < \check{\e}$, then
	\begin{align*}
	 c_{\ref{7c1}} \delta^{\tilde{\beta}-\beta}|n\log{\delta}|^{\frac{n+2}{2}}
	 \leq \frac{1}{4(2\sqrt{4\pi})^{n-1} e^{\frac{1}{16}}}.
	\end{align*}
	If $1/2 < \tilde{\e}$, by Lemma \ref{l1} and \eqref{9e1}, we have
	\begin{align*}
	\frac34 E_1 <l^{n-1} \frac{\tilde{\mu}_1 (B_{\tilde{r}_0})}{\tilde{r}^{n-1}_0} 
	\leq \omega_n l^{n-1} \tilde{r}_0 \left(\tilde{\e} c_{\ref{1c1}}^2 + 2 \sup_{|x|\leq 1} W(x)  \right)
	\end{align*}
	and since $l\tilde\e=\e\leq \frac12$, $l<1$ and $\tilde r_0\leq 1/4$, we obtain
	\begin{align}
	E_1 < \frac43 \omega_n \tilde{r}_0 \left(c_{\ref{1c1}}^2 + 2 \sup_{|x|\leq 1} W(x)  \right). \label{9e2f}
	\end{align}
	If $\hat{\e} \leq \tilde{\e} \leq 1/2$, again by Lemma \ref{l1}, \eqref{9e1}, and 
	$\tilde{r}_0 \leq 1/4$, we have
	\begin{align*}
	\frac{3}{4l^{n-1}} E_1 <& \frac{\tilde{\mu}_1 (B_{\tilde{r}_0})}{\tilde{r}^{n-1}_0} \\
	\leq& \frac{\omega_n \tilde{r}_0}{\tilde{\e}} \left(c_{\ref{1c1}}^2 + 2 \sup_{|x|\leq 1} W(x)  \right)\\
	\leq &  \frac{\omega_n}{4\hat{\e}} \left(c_{\ref{1c1}}^2 + 2 \sup_{|x|\leq 1} W(x)  \right)
	\end{align*}
	and we obtain
	\begin{align}
	E_1 \leq \frac{ \omega_n}{3\hat{\e}} \left(c_{\ref{1c1}}^2 + 2 \sup_{|x|\leq 1} W(x)  \right).
	\label{9e3f}
	\end{align}
	If $\tilde{\e}<\hat{\e}$, integrate \eqref{4e1} for $[1-\hat{t},1]$. 
	Set $s =1 +\tilde{r}_0^2$ and $\tilde{\rho} = \tilde{\rho}_{(0,s)}$.
	we have
	\begin{align}
	&\int_{B_{\frac14}} \tilde{\rho} \ d \tilde{\mu}_1 \nonumber\\ 
	\leq&\int_{B_{\frac14}} \tilde{\rho} \ d \tilde{\mu}_{1-\hat{t}} +
	\frac12 \int^1_{1-\hat{t}} \int_{B_{\frac14}} \tilde{\rho} |u|^2 \ d \tilde{\mu}_\lambda \ d\lambda \nonumber \\
	&+  \int^1_{1-\hat{t}} \frac{1}{2(s-\lambda)} \int_{B_{\frac14}} 
	\left(\frac{\tilde{\e} |\nabla \tilde{\vhi}|^2}{2} -\frac{W(\tilde{\vhi})}{\tilde{\e}}\right)_+
	\tilde{\rho}\ d \tilde{x} d\lambda \nonumber \\
	&+ c_{\ref{4c1}}  \int^1_{1-\hat{t}} \tilde{\mu}_{\lambda}(B_{\frac14}) \  d\lambda \nonumber \\
	=:& I_1+I_2+I_3+I_4. \label{9e7}
	\end{align}
	By \eqref{9e3}, $I_1$ gives
	\begin{align} 
	I_1 \leq \int_{B_{\frac14}} \frac{1}{(\sqrt{4\pi})^{n-1}  (\tilde{r}_0^2 +\hat{t})^{\frac{n-1}{2}}} \ d\tilde{\mu}_{1-\hat{t}} 
	 \leq \frac{\Lambda_1 l^{1-n}}{ (\sqrt{4\pi \hat{t}})^{n-1} } \label{9e8}
	\end{align}
We next consider the term of $I_2$. By \eqref{MTe1}, we have $2-\frac{n}{p}-\frac{2}{q}>0$.
Hence, we obtain
\begin{align} \label{9e7a}
\|\tilde{u}\|_{L^q([0,1];(W^{1,p}(U_{1+\tilde{\e}}))^n)} \leq  \Lambda_0
\end{align}
by $l \leq 1$ and \eqref{9e2a}.
	Using Lemma \ref{l8}, \eqref{9e6}, \eqref{9e6b}, and \eqref{9e7a}, $I_2$ gives
	\begin{align}
	I_2 \leq  \frac{c_{\ref{8c1}}}{2} \hat{t}^{\hat{p}} \Lambda_0^2 \tilde{E} 
	\leq \frac{1}{4(\sqrt{4\pi})^{n-1} e^{\frac{1}{16}}} l^{1-n} E_1. \label{9e9}
	\end{align}
	Using Lemma \ref{l7}, \eqref{9e6}, and \eqref{9e6a}, $I_3$ gives
	\begin{align}
	I_3 \leq \frac{1}{4(2\sqrt{4\pi})^{n-1} e^{\frac{1}{16}}}\left(1+ \tilde{E}\right) \leq \frac{1}{4(\sqrt{4\pi})^{n-1} e^{\frac{1}{16}}}(1 + l^{1-n} E_1). \label{9e10}
	\end{align}
	By \eqref{9e3}, $I_4$ gives
	\begin{align}
	I_4 \leq c_{\ref{4c1}}  \hat{t} \Lambda_1 l^{1-n}.
	\end{align}
	By \eqref{9e1}, we have
	\begin{align}
	\frac{3l^{1-n}}{4} E_1 \leq  \frac{\mu_{t_0} (B_{r_0}(x_0) )}{r_0^{n-1}}
	 = \frac{\tilde{\mu}_1(B_{\tilde{r}})}{\tilde{r}^{n-1}_0}
	 \leq e^{\frac{1}{16}}(\sqrt{4\pi})^{n-1} \int_{B_{\frac14}} \tilde{\rho} \ d \tilde{\mu}_1.
	 \label{9e11}
	\end{align}
	By \eqref{9e7}--\eqref{9e11}, there exists a constant
	 $c_{\cc \label{9c2}}= c_{\ref{9c2}}(n,p,q,\beta,W,\Lambda_0, \Lambda_1)$ such that
	 \begin{align}
	 E_1 \leq c_{\ref{9c2}}, \label{9e4f}
	 \end{align}
	and this theorem follows by \eqref{9e1f}, \eqref{9e2f}, \eqref{9e3f}, and \eqref{9e4f}.
\end{proof}
We redefine $\Omega$ as a general bounded domain with smooth boundary.
\begin{proof}[Proof of Theorem \ref{MT1}]
	Since $\bar{\Omega}' \subset \Omega$ is compact, there exist a constant $\tilde{r}>0$ and sets $\{U(x_i,r_i)=\{ |x-x_i|<r_i\}\}_{i=0}^N$ such that for any $y\in \Omega'$, there exists a set 
	$U(x_i,r_i)$ such that 
	$B_{\tilde{r}}(y) \subset U(x_i,r_i)$ and
	$\bar{\Omega}' \subset \subset \cup_{1\leq i\leq N} U(x_i,r_i) \subset \subset \Omega$.
For sufficiently small $\e>0$, we obtain $\Omega' \times [\tau ,T] \subset \cup_{1\leq i\leq N} U(x_i,r_i - \e) \times (\e^2,T]$. For any $U_r(x) \subset \Omega'$ and $t \geq 2\e^2$, if $r < \tilde{r}$, there exists a constant $0<c_{\cc \label{10c1a}}=c_{\ref{10c1a}}(n,p,q,\beta,\Lambda_0, \Lambda_1, \Omega', \Omega, W)$ such that
\begin{align} \label{1.2t}
\frac{\mu_t(B_r(x))}{r^{n-1}}\leq c_{\ref{10c1a}} (1+t^{\frac{1-n}{2}}),
\end{align}
using Theorem \ref{l9}. By $t\geq \tau$, we have
\begin{align*}
	\frac{\mu_t(B_r(x))}{r^{n-1}}\leq c_{\ref{10c1a}} (1+\tau^{\frac{1-n}{2}}).
\end{align*} If $r \geq \tilde{r}$, we have
\begin{align}\label{1.2t2}
\frac{\mu_t(B_r(x))}{r^{n-1}}\leq \frac{\Lambda_1}{\tilde{r}^{n-1}}.
\end{align}
Thus this theorem follows.
\end{proof}
We prove theorems we will need in Section 3.
\begin{lemm}\label{ll12}
	There exist a constant $c_{\cc \label{10ac1}}=c_{\ref{10ac1}}(n,p,q,\beta,\Lambda_0,\Lambda_1,\Omega', \Omega)$ such that, if $\e<\e_{\ref{ME1}}$ and $2\e^2 \leq t\leq 1$, we have
	\begin{align*}
		\int_{\Omega'} \e |u|^2 |\nabla \vhi|^2 \ dx
		\leq c_{\ref{10ac1}}t^{\iota_1} \|u(\cdot,t)\|_{W^{1,p}(\Omega)}^2,
	\end{align*}
	where 
	\begin{align*}
		\iota_1 = \begin{cases}
			\frac{1-n}{p} &(p\geq 2)\\
			1-\frac{n}{p} &(p<2).
		\end{cases}
	\end{align*}
	Note that $\iota_1$ is a negative constant.
\end{lemm}
\begin{proof}
	Fix $2\e^2\leq t$. Let $\Omega' \subset \subset \Omega'' \subset \subset \Omega$. Let $\phi \in C^2_c(\Omega)$ be a function such that $0\leq \phi \leq 1$, $\phi =1$ on $\Omega'$, and $\phi =0$ on $\Omega \backslash \Omega''$. We first consider the case of $p \geq 2$.
	By Theorem  \ref{4MZ}, Lemma \ref{ll9}, \eqref{1.2t}, and \eqref{1.2t2}, there exists a constant
	$c_{\cc \label{11cb2}}=c_{\ref{11cb2}}(n,p,q,\beta,\Lambda_0,\Lambda_1, \Omega'', \Omega,W, \phi)$ such that
	\begin{align} \label{l12ee1}
		\begin{split}
			&	 \int_{\Omega'} \e |u|^2 |\nabla \vhi|^2 \ dx \\
			\leq& \left( \int_{\Omega'}  \e |u|^p |\nabla \vhi|^2 \ dx \right)^{\frac{2}{p}}
			(2 \Lambda_1)^{1-\frac{2}{p}} \\
			\leq &	(2 \Lambda_1)^{1-\frac{2}{p}} \left( \int_{\Omega''} \phi \e |u|^p |\nabla \vhi|^2 \ dx \right)^{\frac{2}{p}}	\\
			\leq &	c_{\ref{11cb2}} (1+t^\frac{1-n}{2})^\frac{2}{p}  \left( \int_{\Omega}  |u|^p +|u|^{p-1} |\nabla u|\ dx \right)^{\frac{2}{p}} \\
			\leq &	2c_{\ref{11cb2}} t^\frac{1-n}{p}  \left( \int_{\Omega}  |u|^p + |\nabla u|^p\ dx \right)^{\frac{2}{p}} \\
			\leq& c_{\ref{10ac1}}t^{\frac{1-n}{p}} \|u(\cdot,t)\|_{W^{1,p}(\Omega)}^2.
		\end{split}
	\end{align}
We next consider the case of $p<2$.
Set $s= \frac{p(n-1)}{n-p}$.
By \eqref{MTe1}, we obtain $2\leq s$.
Using the H\"{o}lder inequality and the Sobolev inequality, we have
	\begin{align}\label{l12ee2}
		\begin{split}
			&	 \int_{\Omega'} \e |u|^2 |\nabla \vhi|^2 \ dx \\
			\leq& \left( \int_{\Omega'}  \e |u|^s |\nabla \vhi|^2 \ dx \right)^{\frac{2}{s}}
			(2 \Lambda_1)^{1-\frac{2}{s}} \\
			\leq &	(2 \Lambda_1)^{1-\frac{2}{s}} \left( \int_{\Omega''} \phi \e |u|^s |\nabla \vhi|^2 \ dx \right)^{\frac{2}{s}}	\\
			\leq &	c_{\ref{11cb2}} (1+t^\frac{1-n}{2} )^\frac{2}{s} \left( \int_{\Omega}  |u|^s +|u|^{s-1} |\nabla u|\ dx \right)^{\frac{2}{s}} \\
			\leq &	c_{\ref{11cb2}} c(n,p,\Omega) t^\frac{1-n}{s} \left( \int_{\Omega} |u|^s \ dx +\left(\int_{\Omega}|u|^\frac{np}{n-p} \ dx  \right)^{\frac{n-p}{np}
				\frac{n(p-1)}{n-p}}   \left(\int_{\Omega}|\nabla u|^p\ dx \right)^{\frac{1}{p}}  \right)^{\frac{2}{s}} \\
			\leq&  c_{\ref{10ac1}}t^{1-\frac{n}{p}} \|u(\cdot,t)\|_{W^{1,p}(\Omega)}^2.
		\end{split}
	\end{align}
Thus this lemma follows.
\end{proof}
Define
\begin{align*}
	\mu_t(\phi(\Omega)) :=\int_{\Omega} \left(\frac{\e |\nabla \vhi|^2}{2}
	+\frac{W(\vhi)}{\e}\right) \phi \ dx
\end{align*}
for $\phi \in C(\Omega)$.
\begin{lemm}\label{ll13}
	There exist a function $\phi \in C^2_c(\Omega)$ with $\mu_t(\phi(\Omega))\leq \Lambda_1$ and  a constant $c_{\cc \label{10c1}}=c_{\ref{10c1}}(n,p,q,\beta,\Lambda_0,\Lambda_1, \Omega', \Omega,W,\phi)$ such that, if $\e<\e_{\ref{ME1}}$ and $2\e^2 \leq t \leq 1$, we have
	\begin{align*}
		 \int_{\Omega'} \frac{\e}{2} \left(\Delta \vhi -\frac{W'(\vhi)}{\e^2}\right)^2 \ dx
		\leq -\frac{d}{dt} \mu_t(\phi(\Omega)) + c_{\ref{10c1}}(1+t^{\iota_1} \|u(\cdot,t)\|_{W^{1,p}(\Omega)}^2)
	\end{align*}
and if $0\leq t\leq 2\e^2 $, we have
	\begin{align*}
	\int_{\Omega'} \frac{\e}{2} \left(\Delta \vhi -\frac{W'(\vhi)}{\e^2}\right)^2 \ dx
	\leq -\frac{d}{dt} \mu_t(\phi(\Omega)) + c_{\ref{10c1}}(1+\e^{-2\beta}).
\end{align*}
\end{lemm}
\begin{proof}
	Let $\Omega' \subset \subset \Omega'' \subset \subset \Omega$. Let $\tilde{\phi} \in C^2_c(\Omega)$ be a function such that $0\leq \tilde{\phi} \leq 1$, $\tilde{\phi} =1$ on $\Omega'$, and $\tilde{\phi} =0$ on $\Omega \backslash \Omega''$.
	By \eqref{MTe2}, we compute
	\begin{align} \label{l12e1}
		\begin{split}
			&\frac{\partial}{\partial t} \int_{\Omega} \left(\frac{\e |\nabla \vhi|^2}{2}
			+\frac{W(\vhi)}{\e}\right) \tilde{\phi}^2 \ dx \\
			=&\int_{\Omega} \left(\e \nabla \vhi \cdot \nabla \partial_t \vhi
			+\frac{W'(\vhi)}{\e} \partial_t \vhi \right) \tilde{\phi}^2 \ dx \\
			=& \int_{\Omega} \partial_t \vhi \left( \left( -\e \Delta \vhi +\frac{W'(\vhi)}{\e} \right)\tilde{\phi}^2- 
			2\e \tilde{\phi} \nabla \vhi \cdot \nabla \tilde{\phi} \right) \ dx\\
			=& \int_{\Omega} \left(\Delta \vhi-\frac{W'(\vhi)}{\e^2}-u\cdot\nabla \vhi \right)  \left( \left( -\e \Delta \vhi +\frac{W'(\vhi)}{\e} \right)\tilde{\phi}^2- 
			2\e \tilde{\phi} \nabla \vhi \cdot \nabla \tilde{\phi} \right) \ dx\\
			\leq&\int_{\Omega} -\frac{\e}{2} \left(\Delta \vhi -\frac{W'(\vhi)}{\e^2}\right)^2 \tilde{\phi}^2
			+2\e (\nabla \vhi\cdot \nabla \tilde{\phi})^2 +\e(u\cdot \nabla \vhi)^2 \tilde{\phi}^2\\
			&+2\e \tilde{\phi}(u\cdot \nabla \vhi)(\nabla \vhi \cdot \nabla \tilde{\phi}) \ dx\\
			\leq&\int_{\Omega} -\frac{\e}{2} \left(\Delta \vhi -\frac{W'(\vhi)}{\e^2}\right)^2 \tilde{\phi}^2
			+3(\tilde{\phi}^2 + |\nabla \tilde{\phi}|^2)(1+|u|^2)\e |\nabla \vhi|^2 \ dx.
		\end{split}
	\end{align}
	If $\e<\e_{\ref{ME1}}$ and $2\e^2 \leq t\leq 1$, we have
	\begin{align*}
		\int_{\Omega'} \e 
		\left(\Delta \vhi -\frac{W'(\vhi)}{\e^2}\right)^2 \ dx
		\leq -\frac{d}{dt} \mu_t(\tilde{\phi}^2(\Omega)) + c_{\ref{10c1}}(1+t^{\iota_1} \|u(\cdot,t)\|_{W^{1,p}(\Omega)}^2),
	\end{align*}
	by \eqref{a5} and Lemma \ref{ll12}. If $0 \leq t  \leq 2\e^2$, we have
	\begin{align*}
		\int_{\Omega'} \e 
		\left(\Delta \vhi -\frac{W'(\vhi)}{\e^2}\right)^2 \ dx
		\leq -\frac{d}{dt} \mu_t(\tilde{\phi}^2(\Omega)) + c_{\ref{10c1}}(1+\e^{-2\beta}),
	\end{align*}
	by \eqref{a3} and \eqref{a5}.
	Thus this lemma follows with $\phi = \tilde{\phi}^2$.
\end{proof}
We redefine $\eta$ of \eqref{eta} as follows. 
\begin{align*}
	\eta(x) \in C^\infty_c(B_{d}) \ \mathrm{with} \ \eta =1 \ \mathrm{on} \ 
	B_{d/2}, \ 0\leq \eta \leq 1,
\end{align*}
where $d=\min \{d(\partial \Omega, \Omega')/2,1/4\}$.
\begin{theo}
	Under the same assumptions of Theorem \ref{MT1}, if $\e<\e_{\ref{ME1}}$ and for
	$t_0<t_1<s$, $t_0,t_1 \in [\tau,T]$, and $y \in \Omega'$, there exists a constant
	$c_{\cc \label{13c1}}=c_{\ref{13c1}}(n,p,q,\beta,\Lambda_0,\Lambda_1,\tau, \Omega', \Omega)$
	such that
	\begin{align}
		\begin{split}
			&	\left.\int_\Omega \tilde{\rho} \ d\mu^\e_t \right|^{t_1}_{t=t_0}
			+\int^{t_1}_{t_0} \frac{dt}{2(s-t)} \int_{\Omega} \left|\frac{\e|\nabla \vhi|^2}{2}
			-\frac{W(\vhi)}{\e} \right|\tilde{\rho} \ dxdt\\
			\leq&c_{\ref{13c1}}((t_1-t_0)^{\hat{p}}+\e^{\tilde{\beta}-\beta}|n\log{\e}|^{\frac{n+2}{2}}
			+e^{-\frac{1}{c_{\ref{13c1}}(s-t_0)}}(t_1-t_0)).
		\end{split}
	\end{align}
\end{theo}
\begin{proof}
	Theorem \ref{l4}, Lemma \ref{l7}, and Lemma \ref{l8} hold if we change $\eta$
	but change the constants. Using there theorems and Theorem \ref{MT1}, this theorem follows.
\end{proof}
\section{Proof of Theorem \ref{MT2}}
We first construct relevant solutions for \eqref{MTe2} when given $\Omega_0 \subset \subset \Omega$
with $ \chi_{\Omega_0} \in BV(\Omega)$ and 
$g \in L^q([0,\infty);W^{2,p}(\Omega)) \cap W^{1,\infty}([0,\infty);L^\infty(\Omega))$.
Using the standard Extension Theorem, we can know that there exists a function 
$\tilde{g} \in L^q([0,\infty);W^{2,p}(\R^n)) \cap W^{1,\infty}([0,\infty);L^\infty(\R^n))$
such that $\tilde{g}=g$ in $[0,\infty) \times \Omega$.
and
\begin{align} \label{p13e1}
	\begin{split}
		&\|\tilde{g}\|_{L^q([0,\infty);W^{2,p}(\R^n)) }
	+\|\tilde{g}\|_{W^{1,\infty}([0,\infty);L^\infty(\R^n))}\\ \leq& c(n,p,q,\Omega)
	(\|g\|_{L^q([0,\infty);W^{2,p}(\Omega)) }
	+\|g\|_{W^{1,\infty}([0,\infty);L^\infty(\Omega))}).
	\end{split}
\end{align}
Let $\Omega',\Omega'' \subset \R^n$ be bounded domains with smooth boundary and $\Omega_0 \subset \subset \Omega' \subset \subset \Omega \subset \subset \Omega''$.
Using \cite[Theorem 3.42]{LA00}, there exists a sequence $\{\Omega_0^i\}_{i \in \Na} \subset \subset \Omega'$ of open sets with smooth boundaries such that $\chi_{\Omega_0^i} \to \chi_{\Omega_0}$ in $L^1(\Omega'')$ and $\|D \chi_{\Omega_0^i}\| (\Omega'') \to \|D \chi_{\Omega_0}\| (\Omega'') $.
Let $d_i$ be the signed distance function to $\partial \Omega_0^i$ which is
positive inside of $\Omega_0^i$. Let $r_i$ be a constant such that $\{x \in \Omega: \dist (x,\Omega_0^i<
r_i)\}\subset \subset \Omega$ and $d_i \in C^{1}(\{x \in \Omega :
\dist (x, \partial \Omega^i_0) < r_i \})$.
Let $h_i \in C^\infty(\R)$ be a monotone increasing function such that
$h_i(s)=r_i/2$ for $s>2r_i/3$, $h_i(s)=s$ for $0\leq s\leq r_i/2$, $h_i(s)=-h_i(-s)$
for $s<0$, and $h_i'(s)\leq 1$ for $s>0$. Define $\tilde{d_i}(x):=h_i(d_i(x))$ for
$x \in \Omega''$. We choose a sequence of $\e_i>0$ such that
\begin{align}
	\lim_{i\to \infty} \frac{\sqrt{\e_i}}{r_i}=0.
\end{align}
We next let $l: \Omega'' \to [0,1]$ be a smooth function such that $l=1$ on $\Omega'$ and $l=0$ on $ \Omega'' \backslash \Omega$.
Let $\Psi$ be a solution to the following problem
\begin{align*}
	\begin{cases}
		\Psi'' (s) = W'(\Psi(x)), \\
		\Psi(0)=0.
	\end{cases}
\end{align*}
Define
\begin{align*}
(\vhi_{\e_i})_0(x) := l(x) \Psi\left(\frac{\tilde{d}_i(x)}{\e_i} \right)+l(x)-1.
\end{align*}
Since $\tilde{g} \in L^q([0,\infty);W^{2,p}(\R^n)) \cap W^{1,\infty}([0,\infty);L^\infty(\R^n))$,
there exist functions $\{g_{\e_i}\}_{i\in \Na} \subset C^\infty(\Omega''\times [0,\infty))$
such that $\|g_{\e_i}-\tilde{g}\|_{ L^q([0,\infty);W^{2,p}(\Omega'')) }+ 
\|g_{\e_i}-\tilde{g}\|_{ W^{1,\infty}([0,\infty);L^\infty(\Omega''))} \to 0$ and
\begin{align} \label{nablag}
\sup_{\Omega'' \times [0,T_i]} \{|\nabla g_{\e_i}|, \e_i |\nabla^2 g_{\e_i}| \} \leq \e_i^{-\beta},
\end{align}
where $T_i:=i$.
Then we consider the following differential equation,
\begin{align}\label{l31}
\begin{cases}
\partial_t \vhi + \nabla g_{\e_i} \cdot \nabla \vhi = \Delta \vhi -\frac{W'(\vhi)}{\e^2} 
&\mathrm{on} \ \Omega''\times [0,T_i],\\
\vhi=(\vhi_{\e_i})_0 &\mathrm{on} \ \Omega''\times \{0\},\\
\vhi=-1 & \mathrm{on} \ \partial\Omega''\times [0,T_i].
\end{cases}
\end{align}
By the standard parabolic theory, there exists a classical solution, which is represented
 by $\vhi_{\e_i}$. We can check that the function $\nabla g_{\e_i}$ satisfy \eqref{a3} and \eqref{a4}.
 Using the maximum principle, the function $\vhi_{\e_i}$ satisfy \eqref{a2}
 on $\Omega''\times [0,T_i]$.
  By the definition of $(\vhi_{\e_{i}})_0$, we have $\mu^{\e_i}_0(\Omega'') < M_1$, where $M_1$ depends only on $\Omega_0$ and $W$. We show that \eqref{a5} is satisfied.
 \begin{lemm} \label{l3a}
 	Given $0<T<T_i$, there exist constants $c_{\cc \label{c31}}=c_{\ref{c31}}(n,p,q,g,T,\Omega,W)$ and $\e_{\ce \label{e31}}= \e_{\ref{e31}}(g)$ such that, if $\e< \e_{\ref{e31}}$,
 	\begin{align}
 	\sup_{t\in [0,T]}\mu^{\e_i}_t(\Omega'') \leq c_{\ref{c31}}.
 	\end{align}
 \end{lemm}
\begin{proof}
	 Since $\partial_t \vhi_{\e_i}=0$ on $\partial\Omega''\times [0,T]$, we compute
	\begin{align*}
	&\frac{\partial}{\partial t} \int_{\Omega''} \left(\frac{\e_i |\nabla \vhi_{\e_i}|^2}{2}
	+\frac{W(\vhi_{\e_i})}{\e_i} \right) \exp(-g_{\e_i}) \ dx \\
	=&\int_{\Omega''} \left( \e_i \nabla \vhi_{\e_i} \cdot \nabla \partial_t \vhi_{\e_i} +
	\frac{W'(\vhi_{\e_i})}{\e_i} \partial_t \vhi_{\e_i} \right) \exp(-g_{\e_i}) \ dx\\
	&-\int_{\Omega''} \left(\frac{\e_i |\nabla \vhi_{\e_i}|^2}{2}
	+\frac{W(\vhi_{\e_i})}{\e_i} \right) \exp(-g_{\e_i}) \partial_t g_{\e_i} \ dx\\
	=&\int_{\Omega''} \left(-\e_i \Delta \vhi_{\e_i} + \e_i \nabla \vhi_{\e_i}
	\cdot \nabla g_{\e_i} +\frac{W'(\vhi_{\e_i})}{\e_i}  \right) \partial_t \vhi_{\e_i} \exp(-g_{\e_i}) \ dx\\
	&-\int_{\Omega''} \left(\frac{\e_i |\nabla \vhi_{\e_i}|^2}{2}
	+\frac{W(\vhi_{\e_i})}{\e_i} \right) \exp(-g_{\e_i}) \partial_t g_{\e_i} \ dx\\
	=& \int_{\Omega''} -\e_i (\partial_t \vhi_{\e_i})^2 \exp(-g_{\e_i})-
	\left(\frac{\e_i |\nabla \vhi_{\e_i}|^2}{2}
	+\frac{W(\vhi_{\e_i})}{\e_i} \right) \exp(-g_{\e_i}) \partial_t g_{\e_i}\ dx\\
	\leq& \|g_{\e_i}\|_{W^{1,\infty}([0,\infty);L^\infty(\Omega''))}
	\int_{\Omega''} \left(\frac{\e_i |\nabla \vhi_{\e_i}|^2}{2}
	+\frac{W(\vhi_{\e_i})}{\e_i} \right) \exp(-g_{\e_i})\ dx.
	\end{align*}
	For sufficiently small $\e$, we have $\|g_{\e_i}\|_{W^{1,\infty}(\Omega'' \times [0,T])} \leq
	\|\tilde{g}\|_{W^{1,\infty}([0,\infty);L^\infty(\Omega''))} +1$. Using the Gr\"{o}nwall's inequality, we obtain
	\begin{align*}
	& \exp(-\|\tilde{g}\|_{W^{1,\infty}([0,\infty);L^\infty(\Omega''))} -1) \mu^{\e_i}_t(\Omega'') \\
	\leq&\int_{\Omega''} \left(\frac{\e_i |\nabla \vhi_{\e_i}(x,t)|^2}{2}
	+\frac{W(\vhi_{\e_i}(x,t))}{\e_i} \right) \exp(-g_{\e_i}(x,t)) \ dx\\
	\leq& \exp((\|\tilde{g}\|_{W^{1,\infty}([0,\infty);L^\infty(\Omega''))} +1)t)\\
	& \times \int_{\Omega''} \left(\frac{\e_i |\nabla \vhi_{\e_i}(x,0)|^2}{2}
	+\frac{W(\vhi_{\e_i}(x,0))}{\e_i} \right) \exp(-g_{\e_i}(x,0)) \ dx \\
	\leq& \exp((\|\tilde{g}\|_{W^{1,\infty}([0,\infty);L^\infty(\Omega''))} +1)(T+1))M_1.
	\end{align*}
	Thus this lemma follows.
\end{proof}
 Define
\begin{align*}
	w_i := \Phi \circ \vhi_{\e_{i}} \ \mathrm{with} \ \Phi(s):= \sigma^{-1} \int^s_{-1}\sqrt{2W(y)} \ dy,
\end{align*}
where $\sigma =\int^1_{-1}\sqrt{2W(y)} \ dy$ and $\vhi_{\e_i}$ is the solution of \eqref{l31}.
Then there exists a subsequence $\{w_{i_k}\}_{k\in \Na} \subset\{w_i\}_{i\in \Na}$
which converges $L^1_{loc}(\Omega \times (0,\infty))$ and pointwise on $\Omega \times(0,\infty)$, and the limit function is $\vhi$. We show that the function $\vhi$ satisfies Theorem \ref{MT2} (2).
\begin{proof}[Proof of Theorem \ref{MT2} (2-d)] 
For simplicity, we will use $\{w_{k}\}_{k\in \Na}$, $\{\vhi_{k}\}_{k\in \Na}$, $g_k$, and $\e_{k}$ instead of $\{w_{i_k}\}_{k\in \Na}$, $\{\vhi_{\e_{i_k}}\}_{k\in \Na}$, $g_{\e_{i_k}}$, and $\e_{i_k}$.
Since $(1+\vhi_{k}(\cdot,0))/2 \to \chi_{\Omega_0}$, we have $w_k \to \chi_{\Omega_0}$ in $L^1$.
Define $\Omega_t = \{x\in \Omega; \vhi(x,t)=1 \}$. We compute
\begin{align} \label{3ee1}
	\begin{split}
			|\Omega_0 \Delta \Omega_t| &=\int_\Omega |\vhi(x,0)-\vhi(x,t)| \ dx\\
			&= \lim_{k \to \infty} \int_\Omega |w_k(x,0)-w_k(x,t)| \ dx\\
			&\leq \liminf_{k \to \infty} \int_\Omega \int^t_0 | \partial_t w_k(x,s)| \ dsdx\\
			&\leq \sigma^{-1} \liminf_{k \to \infty} \int_\Omega \int^t_0 |\sqrt{2W(\vhi_k)} \partial_t \vhi_k | \ dsdx\\
				&\leq \sigma^{-1} \liminf_{k \to \infty}  \int^t_0\int_\Omega \frac{\e_k |\partial_t \vhi_k|^2}{2}s^{\iota_2}+ \frac{W(\vhi_k)}{\e_k}s^{-\iota_2}  \ dxds,
	\end{split}
\end{align}
where $\iota_2=1/q-\iota_1/2$.
By definition of $p$, $q$, and $\iota_1$, we have $0< \iota_2 <1$. By Lemma \ref{l3a}, we have
\begin{align}
	\begin{split}
	 &\int^t_0\int_\Omega  \frac{W(\vhi_k)}{\e_k}s^{-\iota_2}  \ dxds\\
	\leq& c_{\ref{c31}} \int^t_0  s^{-\iota_2}  \ ds\\
	\leq& \frac{c_{\ref{c31}}}{1-\iota_2} t^{-\iota_2+1}.
	\end{split}
\end{align}
By Lemma \ref{ll13} and Lemma \ref{l3a}, there exists a function $\phi \in C^2_c(\Omega'')$ with $\mu_t^{\e_k}(\phi(\Omega''))\leq c_{\ref{c31}}$ such that, if $\e_k<\e_{\ref{ME1}}$ and $2\e_k^2 \leq t \leq 1$, we have
\begin{align} \label{e3.8}
	\int_{\Omega} \frac{\e_k}{2} \left(\Delta \vhi_k -\frac{W'(\vhi_k)}{\e_k^2}\right)^2 \ dx
	\leq -\frac{d}{dt} \mu_t^{\e_k}(\phi(\Omega'')) + c_{\ref{10c1}}(1+t^{\iota_1} \|\nabla g_k(\cdot,t)\|_{W^{1,p}(\Omega)}^2)
\end{align}
and if $0\leq t\leq 2\e^2 $, we have
\begin{align} \label{e3.9}
	\int_{\Omega} \frac{\e_k}{2} \left(\Delta \vhi_k -\frac{W'(\vhi_k)}{\e_k^2}\right)^2 \ dx
	\leq -\frac{d}{dt} \mu_t^{\e_k}(\phi(\Omega'')) + c_{\ref{10c1}}(1+\e_k^{-2\beta}).
\end{align}
By \eqref{nablag}, \eqref{l31}, \eqref{e3.8}, \eqref{e3.9}, Lemma \ref{ll12}, Lemma \ref{l3a}, and integration by part, we have
\begin{align}\label{3ee2}
	\begin{split}
		&\int^t_0\int_\Omega \frac{\e_k |\partial_t \vhi_k|^2}{2}s^{\iota_2} \ dxds\\
	\leq& \int^{2\e^2_k}_0 s^{\iota_2} \int_\Omega \frac{\e_k|\nabla g_k|^2 |\nabla \vhi_k|^2 }{2}+ \frac{\e_k }{2}
	\left(\Delta \vhi_k -\frac{W'(\vhi_k)}{\e^2_k}\right)^2  \ dxds\\
	&+ \int^t_{2\e^2_k}s^{\iota_2} \int_\Omega \frac{\e_k|\nabla g_k|^2 |\nabla \vhi_k|^2 }{2}+ \frac{\e_k }{2}
	\left(\Delta \vhi_k -\frac{W'(\vhi_k)}{\e^2_k}\right)^2  \ dxds\\
	\leq& \int^{2\e^2_k}_0 s^{\iota_2} \left(c_{\ref{c31}} \e^{-2\beta}- \frac{d}{ds} \mu_s^{\e_k}(\phi(\Omega'')) + c_{\ref{10c1}}(1+\e^{-2\beta})\right) \ ds\\
	&+ \int^t_{2\e^2_k}s^{\iota_2} \left(c_{\ref{10ac1}}s^{\iota_1} \|\nabla g_k(\cdot,s)\|_{W^{1,p}(\Omega)}^2-\frac{d}{ds} \mu_s^{\e_k}(\phi(\Omega'')) + c_{\ref{10c1}}(1+s^{\iota_1} \|\nabla g_k(\cdot,s)\|_{W^{1,p}(\Omega)}^2) \right)  \ ds\\
	\leq & c_{\ref{c13d}}(\e_k^{2\iota_2}+t^{\iota_1+\iota_2+1-\frac{2}{q}}+t^{\iota_2}),
	\end{split}
\end{align}
where $c_{\cc \label{c13d}}=c_{\ref{c13d}}(n,p,q,\beta,g,\Omega_0,\Omega,W)$.
 By \eqref{3ee1}--\eqref{3ee2}, we obtain
 \begin{align*}
 		&|\Omega_0 \Delta \Omega_t|\\ \leq  &\sigma^{-1} \liminf_{k \to \infty} \left(\frac{c_{\ref{c31}}}{1-\iota_2} t^{-\iota_2+1}
 		+c_{\ref{c13d}}(\e_k^{2\iota_2}+t^{\iota_1+\iota_2+1-\frac{2}{q}}+t^{\iota_2}) \right) \\
 		\leq& \frac{\sigma^{-1} c_{\ref{c31}}}{\iota_2+1} t^{-\iota_2+1}+
 		\sigma^{-1} c_{\ref{c13d}} t^{\iota_1+\iota_2+1-\frac{2}{q}}  +
 		\sigma^{-1}c_{\ref{c13d}} t^{\iota_2}.
 \end{align*}
Thus this theorem follows by $-\iota_2+1=\iota_1+\iota_2+1-\frac{2}{q}>0$ and $\iota_2>0$. 
\end{proof}
\begin{proof}[Proof of Theorem \ref{MT2}] Given $0<\tau<T$.
	Using Theorem \ref{MT1} and Lemma \ref{l3a}, there exist constants
	$\e_{\ce \label{e3}}= \e_{\ref{e3}}(n,p,q,g,\tau,T, \Omega, \Omega'',W)>0$ and
	$D_2= D_2(n,p,q,g,\tau, T,\Omega, \Omega'',W)>0$ such that, if
	$\e<\e_{\ref{e3}}$, we have
	\begin{align}\label{3e3}
	\sup_{U_{r}(x)\subset \Omega,t \in [\tau,T]}  \frac{\mu_{t} (B_r(x))}{r^{n-1}} \leq D_2.
	\end{align}
	By \eqref{3e3} and Section 2, we can check that all the theorems
	except Proposition 8.5 and Proposition 8.6 related to a unit density of $V_t$ from 
	Section 5 onwards in \cite{TA16} hold even if we change $\Omega = \R^n$ or $\T^n$ to a bounded domain $\Omega'$ and change the initial time to $\tau$.
	This is because many theorems have been proved on bounded domains, and other theorems, namely Proposition 6.1, Proposition 7.3, Theorem 7.1, Proposition 8.3, and Proposition 8.4 in \cite{TA16}, can be proved exactly the same even if the domain is bounded. Thus this theorem follows.
\end{proof}

\end{document}